\def\draft{n}
\documentclass{amsart}[12pt]

\usepackage{latexsym,amssymb,graphicx,amscd,amsmath}

\def\printname#1{
	\if\draft y
		\smash{\makebox[0pt]{\hspace{-0.5in}
			\raisebox{8pt}{\tt\tiny #1}}}
	\fi
}

\def\lbl#1{\label{#1}\printname{#1}}

\newtheorem{thm}{Theorem}[section]
\newtheorem{lem}[thm]{Lemma}
\newtheorem{prop}[thm]{Proposition}
\newtheorem{cor}[thm]{Corollary} 
 
\newtheorem{rem}[thm]{Remark}
\newtheorem{rems}[thm]{Remarks}
\newtheorem{ex}[thm]{Example}

\newtheorem{question}[thm]{Question}

\newcommand{\BS}{{\mathcal{S}}}

\newcommand{\BZ}{{\mathbb{Z}}}
\newcommand{\BQ}{{\mathbb{Q}}}
\newcommand{\BR}{{\mathbb{R}}}

\newcommand{\BT}{{\mathbb{T}}}

\newcommand{\BC}{{\mathbb{C}}}

\newcommand{\BB}{{\mathcal{B}}}

\newcommand{\BH}{{\mathcal{H}}}

\newcommand{\BL}{{\mathcal{L}}}

\newcommand{\BF}{{\mathbb{F}}}

\newcommand{\Si}{{\Sigma}}

\newcommand{\bH}{{\mathcal{H}}}

\newcommand{\la}{{\lambda}}

\newcommand{\om}{{\omega}}
\newcommand{\si}{{\sigma}}

\newcommand{\al}{\alpha}

\newcommand{\fb}{\frak{b}}

\newcommand{\I}{{\mathrm I}}

\DeclareMathOperator{\End}{End}
\DeclareMathOperator{\Hom}{Hom}

\DeclareMathOperator{\Sp}{Sp}

\DeclareMathOperator{\SL}{SL}

\DeclareMathOperator{\GL}{GL}

\DeclareMathOperator{\SO}{SO}

\DeclareMathOperator{\even}{\, even}
\DeclareMathOperator{\odd}{\, odd}

\def\mapright#1{\smash{\mathop{\longrightarrow}\limits^{#1}}}

\begin{document}

\title[An application of TQFT to modular representation theory]{An application of TQFT to modular representation theory}
 
\author{ Patrick M. Gilmer}
\address{Department of Mathematics\\
Louisiana State University\\
Baton Rouge, LA 70803\\
USA}
\email{gilmer@math.lsu.edu}
\thanks{The first author was partially supported by  NSF-DMS-1311911}
\urladdr{www.math.lsu.edu/\textasciitilde gilmer/}


\author{Gregor Masbaum}
\address{Institut de Math{\'e}matiques de Jussieu (UMR 7586 du CNRS)\\
Case 247\\
4 pl. Jussieu\\
75252 Paris Cedex 5\\
FRANCE }
\email{gregor.masbaum@imj-prg.fr}
\urladdr{webusers.imj-prg.fr/~gregor.masbaum}

\date{April 24, 2017}

\begin{abstract} 
For $p\geq 5$ a prime, and $g\geq 3$ an integer, we use Topological Quantum Field Theory (TQFT) to study 
a family of $p-1$ highest weight modules $L_p(\lambda)$ for the symplectic group $\Sp(2g,K)$ where $K$ is an algebraically closed field of characteristic $p$. 
This permits explicit formulae for the
dimension and the formal character 
 of $L_p(\lambda)$ for these highest weights.
\end{abstract}

\maketitle
\tableofcontents

\section{Introduction} \lbl{sec.intro}

Let  
$p$ be an odd prime,
and $K$ be an algebraically closed field of characteristic $p$. For $g\geq 1$ an integer, we consider the symplectic group 
$\Sp(2g,K)$, 
thought of
as an algebraic group of rank $g$. 
It is well-known that
the classification (due to Chevalley) of rational simple $\Sp(2g,K)$-modules 
 is the same as in characteristic zero
(see  Jantzen  \cite[II.2]{J}).
 More precisely,  for every dominant weight $\lambda$ there is a simple module $L_p(\lambda)$, and these exhaust all isomorphism classes of simple modules. 
Here 
the set of dominant weights is the same as in characteristic zero:  $\lambda$ is dominant iff it is a linear combination of the fundamental weights $\om_i$ ($i=1,\ldots, g$) with nonnegative integer coefficients. 

On the other hand, 
while the dimensions of
simple 
  $\Sp(2g,\BC)$-modules can be computed from the Weyl character formula, it seems that explicit dimension formulae for 
the 
 modules $L_p(\la)$ for $p>0$ are quite rare, except in rather special situations. 
We refer to \cite{Hum} for a survey. For fundamental weights, Premet and Suprunenko \cite{PS} 
 gave an algorithm to compute the dimensions of $L_p(\om_i)$ by reducing 
the problem to known properties of symmetric group representations. 
Later, Gow \cite{G} gave an explicit construction of $L_p(\om_i)$ for the last $p-1$ fundamental weights (that is:  $\om_i$ where $i\geq g-p+1$) which allowed him to obtain a 
recursive formula for their dimensions. Even later, Foulle \cite{F} obtained  a dimension formula for all fundamental weights. 
As for 
 other 
weights, it is
 known  
that for weights $\lambda$ in the fundamental alcove the dimension of $L_p(\lambda)$ is the same as the dimension of $L_0(\la)$ (the corresponding simple module in characteristic zero), and can thus be computed by the Weyl character formula. But for weights outside the fundamental alcove, 
no general dimension formula 
is known. 
A conjectural formula by Lusztig for primes in a certain range was shown to hold for $p>>0$  by Andersen-Jantzen-Soergel \cite{AJS} but was recently  shown not to hold for all $p$ in the hoped-for range  by Williamson \cite{W}.

In this paper, we show that Topological Quantum Field Theory (TQFT) can give new information about 
the dimensions of 
some of these simple modules. Specifically,
we show that for every prime $p\geq 5$ and in every rank $g\geq 3$, there is a family of $p-1$ dominant 
weights $\la$, lying outside of the fundamental alcove except for one weight in rank $g=3$, for which we can express the dimension of $L_p(\la)$ by 
formulae similar to the Verlinde formula in TQFT. We found this family 
as a byproduct of Integral $\SO(3)$-TQFT
 \cite{Gi,GM1},  
an integral refinement of the  
 Witten-Reshetikhin-Turaev  
TQFT associated to $\SO(3)$.
More precisely, we use Integral $\SO(3)$-TQFT in what we call the `equal characteristic case' which we studied  in 
 \cite{GM3}.  The family of weights $\la$ we found together with our formulae for $\dim L_p(\la)$ is given in the following Theorem~\ref{1.1}. We can also compute the 
weight space decomposition of $L_p(\la)$ for these weights $\la$;   this will be given in Theorem~\ref{1.8}.
We follow the notation of \cite[Planche III]{B}, where the fundamental weights $\omega_i$ are expressed in the usual basis $\{\varepsilon_i\}$ ($i=1, \ldots, g$) of weights of the maximal torus as $\omega_i=\varepsilon_1 + \ldots + \varepsilon_i$.

\begin{thm}\lbl{1.1} Let $p\geq 5$ be prime and put $d=(p-1)/2$. For rank $g\geq 3$, consider the following $p-1$ dominant weights for the symplectic group $\Sp(2g,K)$ : \[\la =
\left\{\begin{array}{lll}
(d-1)\om_g &  &\ \ (Case \ \mathrm{I})\\
(d-c-1)\om_g + c \, \om_{g-1}& \mbox{ for $1\le c \le d-1$}  &\ \ (Case \ \mathrm{II})\\
(d-c-1)\om_g + (c-1)\om_{g-1} +  \om_{g-2} & \mbox{ for $1\le c \le d-1$}  &\ \ (Case \ \mathrm{III})\\
(d-2) \om_g +   \om_{g-3}& &\ \ (Case \ \mathrm{IV})\\
\end{array}\right. 
\]   
Put $\varepsilon= 0$ in Case $\mathrm I$ and $\mathrm {II}$ and $\varepsilon= 1$ in Case $\mathrm {III}$ and $\mathrm {IV}$. Then

\begin{equation}\lbl{diff} 
\dim 
L_p(\la)= \frac 1 2 \left(D_g^{(2c)}(p) + (-1)^\varepsilon\delta_g^{(2c)}(p)\right) \quad \text{where} \end{equation}
 \begin{equation}\lbl{Verl} 
D_g^{(2c)}(p)=\left( \frac p 4\right)^{g-1} \sum_{j=1}^{d}
\left(\sin \frac { \pi j(2c+1)}{p}\right) \left( \sin \frac { \pi j}{p} \right)^{1-2g}
\end{equation}  
\begin{equation}\lbl{d-Verl}  \delta_g^{(2c)}(p)=  (-1)^c \frac{4^{1-g}}{p}
  \sum_{j=1}^{d}  \left(\sin\frac{\pi j (2c+1)}{p}\right) \left(\sin\frac{\pi j}{p}\right) \left(\cos\frac{\pi j}{p} \right)^{-2g},
\end{equation}
and $c$ is the same $c$ used in the definition of $\lambda$, except in Case $\mathrm I$ and  $\mathrm {IV}$, where we put $c=0$. In Case $\mathrm{IV}$ in rank $g=3$, $\omega_{g-3}=\omega_0$ should be interpreted as zero.
\end{thm}

\begin{rem}{\normalfont Formula (\ref{Verl}) 
is an instance of the famous Verlinde formula in TQFT.  Formula (\ref{d-Verl}) 
appeared first in \cite{GM3}. Note that the difference between the two formulae is that certain sines in (\ref{Verl}) have become cosines in (\ref{d-Verl}), and the overall prefactor is different. For fixed $g$, both $D_g^{(2c)}(p)$ and $\delta_g^{(2c)}(p)$ can be expressed as polynomials in $p$ and $c$. See \cite{GM3} for more information and further references. In Appendix~\ref{App-B}, we give explicit polynomial expressions for the dimensions of our $L_p(\la)$ in rank $g\leq 4$.
}\end{rem}

 \begin{rem}\lbl{1.6}{\normalfont 
Except for the weight  
      $\la=(d-2)\omega_3$ 
in Case $\mathrm{IV}$ in rank $g=3$, 
all the weights in the above list lie
outside of the fundamental alcove.
See Section~\ref{sec5} for more 
concerning
this.
}\end{rem}

\begin{rem}{\normalfont When $p=5$, the list above produces (in order) the fundamental weights $\om_{g}, \om_{g-1}, \om_{g-2}, \om_{g-3}$. These are exactly the weights considered by Gow \cite{G}.  For $p>5$, our weights are different from those of Gow. It is intriguing that both Gow's and our family of weights have $p-1$ elements.
}\end{rem}

\begin{question}{\normalfont Can one find similar  Verlinde-like dimension formulae for other families of 
dominant weights?
}\end{question}
\begin{rem}{\normalfont In \cite{GM4}, we answered this question affirmatively for the $p-1$ fundamental weights considered by Gow. But \cite{GM4} was based on Gow's recursion formula, not on TQFT as in the present paper. On the other hand, $\SO(3)$-TQFT is just 
one of the simplest TQFTs 
within the family of Witten-Reshetikhin-Turaev TQFTs, and it is conceivable that other Integral TQFTs 
({\em e.g.} \cite{CL2})
might produce more families of 
weights $\la$ where the methods of the present paper could be applied. A difficulty here is that Integral TQFT as we need it in this paper is so far not developed for other  TQFTs.
}\end{rem}

\begin{rems}\lbl{17s}{\normalfont  
(i) The restriction that $g\geq 3$ in the theorem is only to ensure that we get $p-1$ distinct weights. The theorem also holds in rank $g=1$ or $2$ for those weights $\la$ where it makes sense ({\em i.e.}, if no  $\om_i$ with $i<0$ appears in the formula for 
      $\la$) provided $\om_0$ is interpreted as zero.

(ii) Case $\mathrm{I}$ could be amalgamated with Case $\mathrm{II}$ in Theorem~\ref{1.1} by allowing $c$ to be zero in Case $\mathrm{II}$. We chose not to do this because 
Case $\mathrm{I}$ will require special treatment later.
}\end{rems}

Throughout the paper, 
we assume $p\geq 5$
and we use the notation $d=(p-1)/2$.

The construction of the modules $L_p(\lambda)$ goes as follows.  For $0\le c \le d-1$ and $\varepsilon \in \BZ/2$, we construct certain simple modules which we denote by $\widetilde F_p(g,c,\varepsilon)$. Note that there are $p-1$ choices of pairs $(c,\varepsilon)$. The construction of the modules $\widetilde F_p(g,c,\varepsilon)$ is based on results from Integral TQFT obtained in \cite{GM3}. From the TQFT description, we shall compute the dimension and weight space decomposition of $\widetilde F_p(g,c,\varepsilon)$. In particular, we shall compute the highest weight occuring in  $\widetilde F_p(g,c,\varepsilon)$, thereby identifying  $\widetilde F_p(g,c,\varepsilon)$ with one of the $L_p(\la)$ in Theorem~\ref{1.1}.

\begin{figure}[h]
\centerline{\includegraphics[width=1.8in]{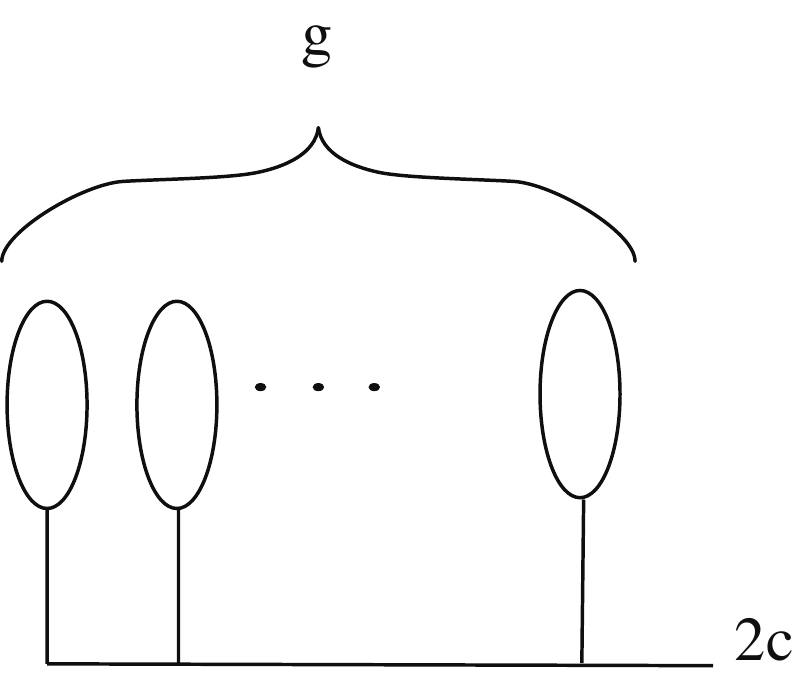}}
\caption{Lollipop tree 
$G_g$}\lbl{lol}
\end{figure}

Here is the construction of $\widetilde F_p(g,c,\varepsilon)$. We give a description which can be read without any knowledge of TQFT. Consider the graph $G_g$ depicted in Figure~\ref{lol} which we call a lollipop tree. It has $2g-1$ trivalent vertices and one univalent vertex which in the figure is labelled $2c$. The $2$-valent `corner'  vertex to the left of the figure should be ignored, and the two edges meeting there are to be considered a single edge.  Thus, $G_g$ has  $3g-1$ edges, $g$ of which are loop edges. The edges incident to a loop edge are called stick edges, and we refer to a loop edge together with its stick edge as a lollipop. 

A $p$-color is an integer $\in\{0,1,\ldots,p-2\}$. A $p$-coloring of $G_g$ is an assignment of $p$-colors to the edges of $G_g$. A $p$-coloring is 
{\em admissible}   
if whenever $i$, $j$ and $k$ are the colors of edges which meet at a vertex, then 
\begin{eqnarray*}\label{adm1} i+j+k &\equiv& 0 \pmod 2~,\\
\label{adm2} |i-j|\ \leq& k &  \leq \ i+j~, \text{\ \ and}\\
\label{adm3} i+j+k &\leq& 2p-4~. \end{eqnarray*} 
Admissibility at the trivalent vertex of the $i$-th lollipop implies that the stick edge has to receive an even color, which we denote by $2a_i$, and the loop edge has to receive a color of the form $a_i+b_i$, with $b_i \geq 0$. We denote the colors of the remaining edges by $c_1,c_2, \ldots$ as in Figure~\ref{lol2}, 
and we write an admissible $p$-coloring as $\sigma= (a_1, \ldots, a_g, b_1, \ldots, b_g, c_1, \ldots )$.

\begin{figure}[h]
 \centerline{\includegraphics[width=1.8in]{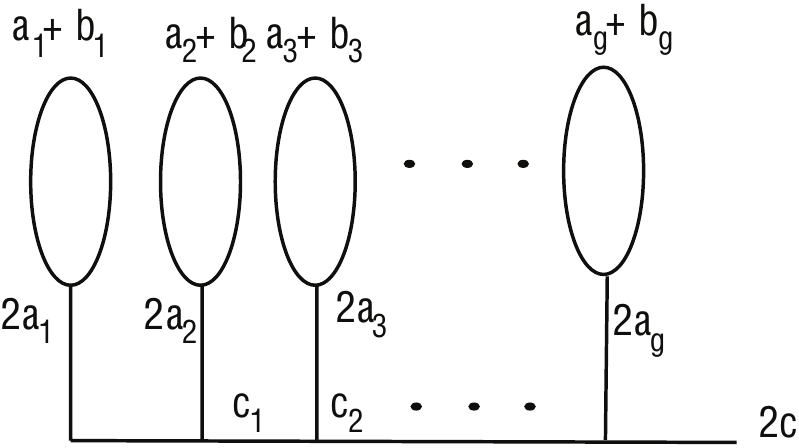}}
\caption{Colored Lollipop tree 
$G_{g}$
}\lbl{lol2}\end{figure}

A 
$p$-coloring
is of type $(c,\varepsilon)$ if the color $2c$ is assigned to the edge incident with the univalent vertex and if  
\begin{equation}\lbl{type}
c+\sum a_i \equiv \varepsilon \pmod 2~.
\end{equation}
A 
$p$-coloring
is 
{\em small}  
if the colors $a_i+b_i$ of the loop edges satisfy 
\begin{equation}\lbl{smallness}0\leq a_i+b_i\leq
d-1~.
\end{equation}

Let $C_p(g,c,\varepsilon)$ denote the set of 
small admissible  
$p$-colorings
of $G_g$ of type  $(c,\varepsilon)$. 
Let  $\BF_p$ denote the finite field with $p$ elements, and 
let $F_p(g,c,\varepsilon)$ be the $\BF_p$-vector space with basis $C_p(g,c,\varepsilon)$.

\begin{thm}\lbl{1.7} There is an irreducible representation of the finite symplectic group $\Sp(2g,\BF_p)$ on $F_p(g,c,\varepsilon)$.
\end{thm}

This 
will be proved in Section~\ref{sec3}.

 Steinberg's restriction theorem (see {\em e.g.} \cite[2.11]{Hum}) 
 implies that there is a unique simple $\Sp(2g,K)$-module $\widetilde F_p(g,c,\varepsilon)$ characterized by the following two properties:

(i) The restriction of $\widetilde F_p(g,c,\varepsilon)$ to the finite group $\Sp(2g,\BF_p)$ is $F_p(g,c,\varepsilon) \otimes K$.

(ii) $\widetilde F_p(g,c,\varepsilon)$ has $p$-restricted highest weight. 

We recall that a dominant weight $\la= \sum_{i=1}^g \la_i \om_i$ is  { \it $p$-restricted} if, for each $1 \le i \le g$, we have
$0 \le  \la_i \le p-1$. 

Part (i) of the following theorem says that the $\widetilde F_p(g,c,\varepsilon)$ are precisely the simple modules $L_p(\lambda)$ listed in Theorem~\ref{1.1}. Part (ii) gives the weight space decomposition and thus determines the formal character of these modules. To state the result, let $\widetilde W_p(g,c,\varepsilon)$ be the multiset of weights occuring in $\widetilde F_p(g,c,\varepsilon)$. (By a multiset, we mean a set with multiplicities.)

\begin{thm} \lbl{1.8}(i) The $\Sp(2g,K)$-module $\widetilde F_p(g,c,\varepsilon)$ is isomorphic to $L_p(\lambda)$ where the highest weight $\lambda=\lambda_p(g,c,\varepsilon)$ is given by 
\begin{align*} \lambda_p(g,0,0) &= (d-1)\omega_g  &&\ \ \ \ \ \ (Case \ \mathrm{I})\\
 \lambda_p(g,c,0) &=(d-c-1)\om_g + c\, \om_{g-1}, &\mbox{$1\le c \le d-1$} &\ \ \ \ \ \ (Case \ \mathrm{II})\\
\lambda_p(g,c,1) &= (d-c-1) \om_g + (c-1) \om_{g-1} +  \om_{g-2}, &\mbox{ $1\le c \le d-1$} &\ \ \ \ \ \ (Case \ \mathrm{III})\\
\lambda_p(g,0,1)&=
(d-2) \om_g +   \om_{g-3} &&\ \ \ \ \ \ (Case \ \mathrm{IV})
\end{align*}

(ii) We have $$\widetilde W_p(g,c,\varepsilon) = \{w(\sigma)\, \vert \, \sigma\in C_p(g,c,\varepsilon)\}~,$$ 
where 
the weight of a coloring $\sigma= (a_1, \ldots, a_g, b_1, \ldots, b_g, c_1, \ldots )$ is 
\begin{equation}\lbl{defw}
w(\sigma)= \sum_{i=1}^g (d-1 -a_i-2 b_i) \varepsilon_i~.
\end{equation}
\end{thm}

\begin{ex}{\normalfont In Case $\mathrm{I}$, the highest weight  corresponds to the $p$-coloring $\sigma_0$ where all edges are colored zero. Indeed, formula (\ref{defw}) gives $$w(\sigma_0)=\sum_{i=1}^g (d-1) \varepsilon_i = (d-1)\om_g~. $$ In the other cases, the coloring $\sigma_0$ is not allowed as it is not of type $(c,\varepsilon)$ for $(c,\varepsilon)\not=(0,0)$. We shall describe the colorings corresponding to the highest weights in Case $\mathrm{II}-\mathrm{IV}$ in Section~\ref{sec4}. 
}\end{ex} 

\begin{rem}{\normalfont The $p=5$ case of Theorem~\ref{1.8} answers affirmatively  the question raised in \cite[p.~ 257 (after Theorem 8.1)]{GM3} (see also \cite[p.~83 (after Corollary 3)]{GM4}). 
}\end{rem}

The remainder of this paper is organized as follows. In Section~\ref{sec2}, we formulate two results (Lemma~\ref{elb1} and Lemma~\ref{elb2}) about the $\Sp(2g,\BF_p)$-modules  $F_p(g,c,\varepsilon)$. In Section~\ref{sec3},  we review the construction of $F_p(g,c,\varepsilon)$ and the proof of Theorem~\ref{1.7}, and then prove Lemma~\ref{elb1} and Lemma~\ref{elb2} using further arguments from TQFT. In Section~\ref{sec4}, we prove Theorems~\ref{1.1} and~\ref{1.8}. The only results from TQFT that will be used in the proof of these two theorems are those stated in Section~\ref{sec2}. Finally, in Section~\ref{sec5}, we make a few further comments and discuss the rank 3 case as an example.

\medskip\noindent{\bf Acknowledgements.} We thank Henning H. Andersen for helpful discussions. He  suggested checking our results against the Jantzen Sum Formula in the rank $3$ case (see Section~\ref{sec5}) and showed us how to do it.  G. M. thanks the Mathematics Department of Louisiana State University, Baton Rouge,   the Centre for Quantum Geometry of Moduli Spaces, Aarhus, Denmark, and the Max Planck Institute for Mathematics, Bonn, Germany, for hospitality while part of this paper was written.   
P. G. 
also thanks  the Max Planck Institute for Mathematics for hospitality. Last but not least, we thank the referee for his insightful comments.
\section{Two lemmas}\lbl{sec2} 

We begin by fixing some notation. 
For $k$ any of the rings $\BZ$, $\BF_p$, or $K$, we take $\Sp(2g,k)$ to be the subgroup of $\GL(2g,k)$ consisting of  isometries of  the  skew symmetric form given by the matrix 
$J_g =\left[\begin{smallmatrix} 0&I_g \\ -I_g&0 \end{smallmatrix} \right]$. 
Let $\BT$ be the maximal torus of $\Sp(2g,K)$ given by the diagonal matrices of $\Sp(2g,K)$. For $1\le i \le g$ and $ x \in K^*$,  let $T_{x, i}$ denote the diagonal matrix with $x$ on the $ i$th  diagonal entry, $x^{-1}$ on the $(g+ i)$-th  diagonal entry, and 
$1$'s 
elsewhere on the diagonal. We have an isomorphism $$(K^*)^g \mapright\approx \BT, \ \ \ \ (x_1, \ldots x_g) \mapsto \prod_{i=1}^g T_{x_i,i }~.$$ We denote by $\{\varepsilon_i\}_{i=1, \ldots, g}$ the standard basis of the weight lattice $$ X(\BT)= \Hom( \BT, K^*) \approx \bigoplus_{i=1}^g \Hom(K^*,K^*) \approx \BZ^g$$ where $\varepsilon_i(T_{x,i})=x$ and  $\varepsilon_i(T_{x,j})=1$ for $j\neq i$. 

We also let $\BB(K)$ denote the Borel subgroup of $\Sp(2g,K)$ which is the group of block matrices of the form
$$\begin{bmatrix} A&B \\ 0&(A^t)^{-1} \end{bmatrix}$$ where $A$ is  an invertible upper triangular matrix and $B$ satisfies $A B^t= B A^t.$

 Recall that $F_p(g,c,\varepsilon)$ is a representation of the finite symplectic group $\Sp(2g,\BF_p)$ on  the $\BF_p$-vector space with basis $C_p(g,c,\varepsilon)$. Let $\hat \fb_\sigma$ denote the basis vector correponding to the coloring $\sigma\in C_p(g,c,\varepsilon)$. Since the finite field $\BF_p$ is a subfield of $K$, we may consider the actions of the finite maximal torus 
  $\BT(\BF_p)=\BT\cap \Sp(2g,\BF_p)$
  and of the finite Borel subgroup 
  $\BB(\BF_p)=\BB(K) \cap  \Sp(2g,\BF_p)$
   on $F_p(g,c,\varepsilon)$. The following Lemma says that the basis vectors  $\hat \fb_\sigma$ are in some sense `weight vectors' for $\BT(\BF_p)$.

\begin{lem}\lbl{elb1} Each basis vector $\hat \fb_{\sigma}$ is a simultaneous eigenvector for the commuting operators $\{T_{x, i}\}_{i=1, \ldots, g}$ for $x\in\BF_p^{\star}$, with eigenvalues given by $$T_{x, i}(\hat \fb_\sigma) = x^{d-1-a_i-2b_i} \hat \fb_{\sigma}$$ where $\sigma= (a_1, \ldots, a_g, b_1, \ldots, b_g, c_1, \ldots )$.
\end{lem}

Note that the exponents are meaningful only modulo $p-1$, as $\BF_p^{\star}$ is a cyclic group of order $p-1$. We may interpret the collection of these exponents for a given basis vector $\hat \fb_{\sigma}$ as specifying a reduced weight, by which we mean an  element of 
$$ X(\BT(\BF_p))= X(\BT)\otimes \BZ/(p-1)\BZ \approx (\BZ/(p-1)\BZ)^g ~.$$ 
Let $W_p(g,c,\varepsilon)$ be the multiset of reduced weights occuring in $F_p(g,c,\varepsilon)$. The following is an immediate corollary of Lemma~\ref{elb1}. 

\begin{cor}\lbl{2.2} We have $W_p(g,c,\varepsilon) = \{\overline w(\sigma)\, \vert \, \sigma\in C_p(g,c,\varepsilon)\}~,$ where $\overline w(\sigma)$ is the reduction modulo $p-1$ of $w(\sigma)$ as defined in (\ref{defw}).
\end{cor}

\begin{rem}\lbl{2.3}{\normalfont Note that $W_p(g,c,\varepsilon)$ is also the reduction modulo $p-1$ of $\widetilde W_p(g,c,\varepsilon)$. This is because the restriction of the  $\Sp(2g,K)$-module $\widetilde F_p(g,c,\varepsilon)$ to the finite group $\Sp(2g,\BF_p)$ is $F_p(g,c,\varepsilon) \otimes K$.
}\end{rem}

 In Section~\ref{sec4}, we shall see that this information  is enough to determine $\widetilde W_p(g,c,\varepsilon)$, except that in Case $\mathrm{I}$, we will also need to use the following Lemma.

\begin{lem}\lbl{elb2} Let  $\sigma_0=(0,0,\ldots)$ be the coloring where all edges are colored zero. Then the basis vector $\hat \fb_{\sigma_0}$  in $F_p(g,0,0)$ is fixed up to scalars by the finite Borel subgroup  $\BB(\BF_p)$.
\end{lem}

The proofs of Lemma~\ref{elb1} and Lemma~\ref{elb2} will be given in Section~\ref{sec3}.

In Section~\ref{sec4} we shall apply Lemma~\ref{elb2} through the following Corollary whose proof we give already here. 

\begin{cor}\lbl{2.4} The highest weight of $\widetilde F_p(g,0,0)$ is congruent modulo $p-1$ to  $w(\sigma_0)=(d-1)\omega_g$.
\end{cor}
\begin{proof}   Let $v$ be a highest weight vector in $\widetilde F_p(g,0,0)$. Then $v$ is fixed up to scalars by  $\BB(K)$ (see \cite[31.3]{H}). Restricting to the finite symplectic group, we can view $v$ as a vector in $F_p(g,c,\varepsilon) \otimes K$ 
 that 
 is fixed
up to scalars by  $\BB(\BF_p)$. By \cite[Theorem 7.1]{CL}, there is a unique line fixed by $\BB(\BF_p)$ in $F_p(g,c,\varepsilon) \otimes K$. Since $\hat \fb_{\sigma_0}$ is also contained in this line by Lemma~\ref{elb2}, we conclude that $v$ and $\hat \fb_{\sigma_0}$ are proportional. In particular, $v$ and $\hat \fb_{\sigma_0}$ have the same reduced weight, which implies the result.\end{proof}

\section{Results from TQFT}\lbl{sec3}

In this section, we review how Integral TQFT leads to the irreducible $\Sp(2g,\BF_p)$-representations $F_p(g,c,\varepsilon)$ of Theorem~\ref{1.7} which were the starting point for this paper. 
In particular, we show how Theorem~\ref{1.7} follows from \cite{GM3} using a result (Lemma~\ref{lemm-app}) originally proved in \cite{M}. We shall  provide a self-contained proof of Lemma~\ref{lemm-app} in Appendix~\ref{App-A}.
We then prove Lemma~\ref{elb1} and Lemma~\ref{elb2}. 

 Let $\Si_g(2c)$ denote a closed surface of genus $g$ equipped with one marked framed point labelled $2c$, where $c$ is an 
integer with $0\leq c\leq d-1$. (Recall $d=(p-1)/2$.) Integral TQFT \cite{GM1} associates to $\Si_g(2c)$ a free $\BZ[\zeta_p]$-module  $\BS_p(\Si_g(2c))$ of finite rank, together with a projective-linear representation of the mapping class group of $\Si_g(2c)$ on this module. Here 
$p\geq 5$ is a prime,
$\zeta_p$ is a primitive $p$-th root of
unity, and $\BZ[\zeta_p]$ is the ring of cyclotomic integers. The mapping class group of $\Si_g(2c)$ can be identified with $\Gamma_{g,1}$, that is, the mapping class group of 
$\Si_{g,1}$,  
an
oriented surface of genus $g$ with one boundary component. (Thus $\Gamma_{g,1}$ is the group of orientation-preserving 
diffeomorphisms
 of 
$\Si_{g,1}$ 
that fix the boundary pointwise, modulo isotopies of such diffeomorphisms.) The projective-linear representation of $\Gamma_{g,1}$ on $\BS_p(\Si_g(2c))$ can be lifted to a linear representation of a certain central extension of $\Gamma_{g,1}$.  The representations of mapping class groups obtained in this way may be considered as an integral refinement of the complex unitary representations coming from Witten-Reshetikhin-Turaev TQFT
associated to the Lie group $\SO(3)$. In particular, the rank of the free $\BZ[\zeta_p]$-module $\BS_p(\Si_g(2c))$ is given by the Verlinde formula (\ref{Verl}).

Recall that $1-\zeta_p$ is a prime in $\BZ[\zeta_p]$, and $\BZ[\zeta_p]/(1-\zeta_p)$ is the finite field
$\BF_p$. Thus we get a representation on the $\BF_p$-vector space 
$$F_p(\Si_g(2c))=\BS_p(\Si_g(2c))/(1-\zeta_p)\BS_p(\Si_g(2c))~.$$ It is shown 
in \cite[Cor. 12.4]{GM-Maslov} that this induces a linear representation of $\Gamma_{g,1}$ on $F_p(\Si_g(2c))$ ({\em i.e.} the central extension is no longer needed). Furthermore, we proved in \cite{GM3} that $F_p(\Si_g(2c))$ has a  
 composition series with (at most) two irreducible factors. These irreducible factors are the $F_p(g,c,\varepsilon)$ defined in the introduction.  More precisely, we have a short sequence of $\Gamma_{g,1}$-representations  
\begin{equation}\lbl{ses} 0 \rightarrow F_p(g,c,1)\rightarrow F_p(\Si_g(2c)) \rightarrow F_p(g,c,0)\rightarrow 0~. 
\end{equation}

It remains to show that the action of $\Gamma_{g,1}$ on the irreducible factors $F_p(g,c,\varepsilon)$ factors through an action of the finite symplectic group $\Sp(2g,\BF_p)$. For $g=1$, this was proved by explicit computation in \cite{GM2}. For $g\geq 2$, we use the following lemma whose proof is deferred to Appendix~\ref{App-A}. 

\begin{lem}\lbl{lemm-app} The Torelli group ${\mathcal I}_{g,1}$ acts trivially on $F_p(g,c,0)$ and  $F_p(g,c,1)$.
\end{lem}

It follows that the action of $\Gamma_{g,1}$ on the irreducible factors $F_p(g,c,\varepsilon)$ factors through an action of the symplectic group $$ \Sp(2g,\BZ) \cong \Gamma_{g,1} / {\mathcal I}_{g,1}~. $$ To see that  this descends to an action of the finite symplectic group $\Sp(2g,\BF_p)$, we invoke a result of Mennicke, who proved that for $g\geq 2$, the group  $\Sp(2g,\BF_p)$ is the quotient of $ \Sp(2g,\BZ)$ by the normal subgroup generated by the $p$-th power of a certain transvection \cite[Satz~10]{Menn}. The result follows, because  transvections lift to Dehn twists in $\Gamma_{g,1}$, and it is well-known that in $\SO(3)$-TQFT at the prime $p$ the $p$-th power of any Dehn twist acts trivially.
This 
concludes the proof of Theorem~\ref{1.7}.

For the proof of  Lemma~\ref{elb1} and  Lemma~\ref{elb2}, we need to say more about the basis vectors $\hat \fb_{\sigma}$ associated to colorings $\sigma$. Recall the graph $G_g$ depicted in Figure~\ref{lol}. A regular neighborhood in $\BR^3$ of $G_g$ is a $3$-dimensional handlebody $\bH_g$. We identify $\Si_g(2c)$ with the boundary of $\bH_g$, in such a way that the univalent vertex labelled $2c$ in the figure meets the boundary surface in the marked point. Given this identification, there is a basis $\{\tilde \fb_\sigma\}$ of $\BS_p(\Si_g(2c))$ called the orthogonal lollipop basis (see \cite[p.~101]{GM2}). The basis vectors are indexed by colorings $\sigma$ in $C_p(g,c,0) \cup C_p(g,c,1)$. Reducing modulo $1-\zeta_p$, we get a basis $\{\hat \fb_\sigma\}$ of $F_p(\Si_g(2c))$. Notice that as an $\BF_p$-vector space,  $F_p(\Si_g(2c))$ is the direct sum of $F_p(g,c,0)$ and $F_p(g,c,1)$. The basis vectors $\hat \fb_\sigma$ where $\sigma\in C_p(g,c,\varepsilon)$ are a basis of $F_p(g,c,\varepsilon)$. 

\begin{rem}\lbl{3.2new}{\normalfont What we denote now by $\BS_p$ was previously denoted by  $\BS_p^+$ in \cite{GM2}, and by $\BS$ in \cite{GM3}. Similarly, in \cite{GM3}, we omitted the subscript $p$ in $F_p(\Si_g(2c))$. Also, we referred to $\varepsilon$ as the parity (even or odd) of a coloring.   Thus $F_p(g,c,1)$ was denoted by $F^{\odd}(\Si_g(2c))$ since it is spanned by the basis vectors corresponding to odd colorings ($\varepsilon=1$). The short exact sequence (\ref{ses}) identifies $F_p(g,c,0)$ with the quotient representation $F_p(\Si_g(2c))/F^{\odd}(\Si_g(2c))$. As a vector space, this quotient was denoted  by $F^{\even}(\Si_g(2c))$ in \cite{GM3} since it is spanned by the basis vectors corresponding to even colorings ($\varepsilon=0$). 
}\end{rem}

\begin{rem}\lbl{3.3}{\normalfont The construction of Integral TQFT in \cite{GM1} uses the skein-theoretic approach to TQFT of \cite{BHMV}. In particular, the basis vectors $\tilde \fb_\sigma$ are represented by certain skein elements (that is, linear combinations of banded links or graphs) in the handlebody $\bH_g$. If a diffeomorphism $f$ of the surface $\Si_g(2c)$ extends to a diffeomorphism $F$ of the handlebody, then the projective-linear action of $f$ on $\BS_p(\Si_g(2c))$ is determined by how $F$ acts on skein elements in the handlebody. The details of  
this are irrelevant for our purposes, with one exception: In the case when $c=0$, the basis vector $\tilde \fb_{\sigma_0}$ associated to the zero coloring $\sigma_0$ can be represented by the empty link in the handlebody $\bH_g$. In particular, it is preserved by any diffeomorphism of the handlebody. So if $f$ extends to a diffeomorphism of the handlebody, then the action of $f$ on $\BS_p(\Si_g(2c))$ fixes (projectively) the basis vector $\tilde \fb_{\sigma_0}$. This will be used below in the proof of Lemma~\ref{elb2}.}\end{rem}

Before giving the proof of  Lemma~\ref{elb1} and  Lemma~\ref{elb2}, we also need to fix our conventions for the homomorphism  $\Gamma_{g,1}\twoheadrightarrow \Sp(2g,\BZ)$. This homomorphism comes from the action of the mapping class group $\Gamma_{g,1}$ on the homology of the surface $\Si_g(2c)$ by isometries of the intersection form. Recall that we  have identified $\Si_g(2c)$ with the boundary of a regular neighborhood $\bH_g$ of the graph $G_g$. We  identify $H_1(\Si_g(2c);\BZ)$ with $\BZ^{2g}$ identifying the homology class of  the positive meridian of  the  $i$th loop (oriented counterclockwise and counting from the left to right) with the $i$th basis vector of $\BZ^{2g}$ and denote this element by $m_i$.  
Similarly we identify the homology class of  the  parallel to the  $i$th loop to be the $(g+i)$th  
basis vector
of $\BZ^{2g}$.
Then the intersection pairing is described by the matrix 
$J_g =\left[\begin{smallmatrix} 0&I_g \\ -I_g&0 \end{smallmatrix} \right]$.

\begin{proof}[Proof of Lemma~\ref{elb1}] First, let us prove Lemma~\ref{elb1} in the special case when $g=1$. Then $F_p(1,c,1)$ is zero and $F_p(1,c,0)=F(\Si_1(2c))$ has dimension $d-c$, as the graph $G_1$ is just a single lollipop, with stick color $2a_1$ equal to $2c$, so that only the color $a_1+b_1$ of the loop edge may vary, and there are $d-a_1=d-c$ possibilities for $b_1$. The representation of  $\SL(2,\BF_p)=  \Sp(2,\BF_p)$ on $F_p(1,c,0)$ is shown in \cite[\S5]{GM2} to 
be 
isomorphic to the standard representation of  $\SL(2,\BF_p)$ on homogeneous polynomials of degree $d-c-1$ in two variables, say $X$ and $Y$. Explicitly, this representation is given by:
\[\begin{bmatrix} \frak a & \frak b\\ \frak c &\frak  d
\end{bmatrix}  X^{d-c-1-b} Y^b  =  (\frak  aX+ \frak cY)^{d-c-1-b}  (\frak bX+\frak dY)^b.\]  
Note that 
\[\begin{bmatrix} x & 0\\ 0 & x^{-1}
\end{bmatrix}  X^{d-c-1-b} Y^b  =  x^{d-1-c-2b}  X^{d-c-1-b} Y^b.\]  
Thus $X^{d-c-1-b} Y^b$, which is the $b$-th element in the monomial basis for the polynomials,  is an eigenvector  for $\left[\begin{smallmatrix} x & 0\\ 0 & x^{-1}
\end{smallmatrix}\right]$ with eigenvalue $ x^{d-1-c-2b}$. One can check that the intertwiner $\Phi$ \cite[\S5]{GM2} defining the isomorphism sends 
$X^{d-c-1-b} Y^b$ to a multiple of $\hat \fb_{\sigma}$ for $\sigma$ the coloring which is $2c$ on the stick edge and $c+b$ on the loop edge of $G_1$. Thus Lemma~\ref{elb1} holds when $g=1$.

The general case is now proved as follows. The torus $\BT(\BF_p)$ is contained in  the subgroup of  $ \Sp(2g,\BF_p)$ isomorphic to a product of $g$ copies of $\SL(2,\BF_p)= \Sp(2,\BF_p)$   arising from the $g$ copies of $\SL(2,\BZ)= \Sp(2,\BZ)$ in $\Sp(2g,\BZ)$ corresponding to each loop of $G_g$. Specifically, the element $T_{x,i}$ defined in Section~\ref{sec2} lies in the $i$-th copy. Similarly, the mapping class group $\Gamma_{g,1}$ contains a subgroup isomorphic to a product of $g$ copies of $\Gamma_{1,1}$. The $i$-th copy of $\Gamma_{1,1}$ is generated by the Dehn twist about the meridian and the Dehn twist about the parallel to the $i$th loop
of $G_g$. 
Since the homomorphism
 $\Gamma_{g,1}\twoheadrightarrow \Sp(2g,\BF_p)$ 
is surjective, we can lift $T_{x,i}$ (non-uniquely) to a mapping class, say  $\phi$,  
which we may assume 
to lie in the $i$-th copy of $\Gamma_{1,1}$. This copy of $\Gamma_{1,1}$ is the mapping class group of the one-holed torus which is cut off from $\Si_g(2c)$ by the simple closed curve $\gamma$ on $\Si_g(2c)$ which is a meridian of the $i$-th stick edge of the graph $G_g$.  Using the integral modular functor properties of \cite[Section 11]{GM1}, we have an injective linear map 
\begin{equation}\label{IF2}
\bigoplus_{a_i=0}^{d-1} \BS_p(\Si_1(2a_i)) \otimes \BS_p(\Si_{g-1}(2a_i,2c))
 \longrightarrow
\BS_p(\Si_g(2c))
\end{equation} given by gluing along $\gamma$. Here, $\Si_{g-1}(2a_i,2c)$ stands for a genus $g-1$ surface with two marked points labelled $2a_i$ and $2c$, respectively. The module  $\BS_p(\Si_{g-1}(2a_i,2c))$ is again a free $\BZ[\zeta_p]$-lattice by \cite[Theorem 4.1]{GM1}. The image of the gluing map (\ref{IF2}) is a free sublattice of $\BS_p(\Si_g(2c))$ of full rank. When tensored with the quotient field of $\BZ[\zeta_p]$, the map (\ref{IF2}) becomes an isomorphism familiar in TQFTs defined over a field under the name of `factorization along a separating curve'. Over the ring $\BZ[\zeta_p]$ the gluing map (\ref{IF2}) is, however,  not surjective in general.

 On the sublattice  of $\BS_p(\Si_g(2c))$ given by the image of the map (\ref{IF2}), the mapping class $\phi$ preserves the direct sum decomposition and in each summand,  $\phi$ acts only on the first tensor factor $\BS_p(\Si_1(2a_i))$. When reduced modulo $1-\zeta_p$, the action induced by $\phi$ on $F_p(\Si_1(2a_i))$ is as described in the genus one case. In particular, 
  the lollipop with stick color $2a_i$ and loop color $a_i + b_i$  
 indexes 
  an eigenvector  for the induced action of $\phi$ 
on $F_p(\Si_1(2a_i))$ 
with eigenvalue $x^{d-1-a_i-2b_i}$. If the map (\ref{IF2}) were an isomorphism of $\BZ[\zeta_p]$-modules, this would prove the lemma by familiar TQFT arguments, since it would then also induce an isomorphism when reduced modulo $1-\zeta_p$. 
 
 Although (\ref{IF2}) is not an isomorphism of $\BZ[\zeta_p]$-modules, we are saved by the following fact (see \cite[Theorem 11.1]{GM1}). Pick a basis $\{ {\mathfrak b}_\nu^{(a_i)}\}$  of the lattice $\BS_p(\Si_{g-1}(2a_i,2c))$
associated to a lollipop tree as in \cite[Theorem 4.1]{GM1}.
 Then the image under the map (\ref{IF2}) of the direct summand $\BS_p(\Si_1(2a_i)) \otimes {\mathfrak b}_\nu^{(a_i)}$ of the L.H.S. of (\ref{IF2})  is a certain power of $1-\zeta_p$ times a direct summand of  the R.H.S. of (\ref{IF2}), that is, of $\BS_p(\Si_g(2c))$. (The power of $1-\zeta_p$ may depend on the summand.) Thus the action of $\phi$ on this direct summand, and hence the action of $T_{x, i}$ on the reduction modulo $1-\zeta_p$ of this direct summand, can be computed from the action of $\phi$ on $\BS_p(\Si_1(2a_i)) \otimes {\mathfrak b}_\nu^{(a_i)}\simeq\BS_p(\Si_1(2a_i))$. Since this action is given by the genus one case, where the lemma is already proved,  it follows that for a coloring  $\sigma= (a_1, \ldots, a_g, b_1, \ldots, b_g, c_1, \ldots )$, the basis vector $\hat\fb_{\sigma}$ is an eigenvector of $T_{x, i}$ with eigenvalue $x^{d-1-a_i-2b_i}$. This completes the proof.
\end{proof}

\begin{rem}\lbl{cauti}{\normalfont  As a word of caution, we mention that a basis  $\{ {\mathfrak b}_\nu^{(a_i)}\}$ of the lattice $\BS_p(\Si_{g-1}(2a_i,2c))$ as needed in the proof above cannot be obtained from colorings of the graph obtained from $G_g$ by cutting $G_g$ at the mid-point of the stick edge of the $i$-th lollipop and removing the connected component containing the loop edge of the lollipop, as one would do when working with TQFTs defined over a field. This is because the remaining graph would not be a lollipop tree. See \cite[Section 10]{GM1}.
}\end{rem}

We now prepare the way for the proof of Lemma~\ref{elb2}. Let $\frak L$ be the span of the homology classes of the meridians  $m_1,m_2,\ldots, m_g$ in  $H_1(\Si_g)=\BZ^{2g}$. 
 Note that $\frak L$ is a lagrangian subspace  with respect to the form $J_g$.   
Let $\BL(\BZ)$ be the subgroup of $\Sp(2g,\BZ)$ consisting of the matrices which preserve $\frak L$. One has that $\BL(\BZ)$ is the set  of matrices of the form
$\left[ \begin{smallmatrix} A&B \\ 0&(A^t)^{-1} \end{smallmatrix} \right]$ where $A \in \GL(g,\BZ)$, and $B$ satisfies $AB^t = BA^t .$ 
We call this subgroup the lagrangian subgroup.

 Let $\Gamma_{g}$ be the mapping class group of the closed surface $\Si_g$ of genus $g$, viewed as the boundary of the handlebody $\BH_g$. Note that  
$\frak L$ is the kernel
 of the map
      $H_1(\Sigma_g) \rightarrow H_1(\bH_g)$.
 If $f \in \Gamma_{g}$, and $f$ extends to a diffeomorphism $F:\bH_g \rightarrow \bH_g$
then  
$f_* \in \BL(\BZ)$.
We have a converse:

\begin{prop} \lbl{Lextend} If 
$f \in \BL(\BZ)$, 
then $f$ is induced by an element of  $\Gamma_{g}$ which extends to a diffeomorphism $F:\bH_g \rightarrow \bH_g$. 
\end{prop}

\begin{proof} Consider  the special case when 
$f \in \BL(\BZ)$
has the form
$\left[\begin{smallmatrix} I_g&B \\ 0&I_g \end{smallmatrix} \right]$.
It follows that $B=B^t.$ Consider the building blocks $\left[\begin{smallmatrix}  I_g&E(i,j) \\ 0&I_g \end{smallmatrix} \right]$ where $E(i,i)$ has zero entries everywhere except for 
the $(i,i)$ location where it has a $1$, and $E(i,j)$ (for $i \ne j$) has zero entries everywhere except for   the $(i,j)$ location and  the $(j,i)$ location where it has  $1$'s. These $E(i,j)$ are realized by   Dehn twists along $m_i$ in the case $i=j$, and along a curve representing $m_i +m_j$ when $i \ne j$.
These curves may be chosen so that they bound disks in $\bH_g$. Thus these Dehn twists extend over $\bH_g$. Products of such Dehn twists realize any symmetric matrix $B$. See \cite[p 312-313]{GL}.

We can reduce the general case  to the above case, using another special case: 
$f \in \BL(\BZ)$
has the form $\left[\begin{smallmatrix}  A&0 \\ 0&D \end{smallmatrix} \right]$. We note that this is  the case when $A=(D^t)^{-1},$ where $D\in \GL(g,\BZ)$.   An elementary matrix in  $\SL(g,\BZ)$ can  be realized, as $D$, by sliding one 1-handle in $\BH_g$ over another. Permuting two handles  realizes, as $D$, a transposition matrix. Any $D\in \GL(g,\BZ)$ is a product of elementary matrices and perhaps a transposition matrix. Thus $\left[\begin{smallmatrix}  (D^t)^{-1}&0 \\ 0&D \end{smallmatrix} \right]$  for any  $D\in \GL(g,\BZ)$ can be realized by a diffeomorphism which extends over $\BH_g$.
\end{proof}

We let $U(\BF_p)$ denote the unipotent radical of the finite Borel subgroup $\BB(\BF_p)$.

\begin{prop} \lbl{LZF}  The image of 
$\BL(\BZ)$
under the quotient map $\pi:\Sp(2g,\BZ)\twoheadrightarrow \Sp(2g,\BF_p)$ contains $U(\BF_p)$. 
\end{prop}

\begin{proof} We have that $U(\BF_p)$ is the group of block matrices over $\BF_p$ of the form $\left[\begin{smallmatrix} V&B \\ 0&(V^t)^{-1} \end{smallmatrix} \right]$ where $V$ is  an invertible upper triangular matrix with $1$'s on the diagonal and $B$ satisfies $V B^t= B V^t.$ Each such matrix may be factored $\left[\begin{smallmatrix} I_g& B V^t  \\ 0&I_g \end{smallmatrix} \right] \left[\begin{smallmatrix} V&0 \\ 0 &(V^t)^{-1}\end{smallmatrix} \right],$ and $B V^t$ will equal its transpose. As above, we note that $\left[\begin{smallmatrix} I_g& B V^t  \\ 0&I_g \end{smallmatrix} \right]$ can be written as a product of $\left[\begin{smallmatrix} I_g&E(i,j) \\ 0&I_g \end{smallmatrix} \right]$ matrices. Thus any matrix of the form 
$\left[\begin{smallmatrix} I_g& B V^t  \\ 0&I_g \end{smallmatrix} \right]$
has lifts under  the quotient map $\pi:\Sp(2g,\BZ)\twoheadrightarrow \Sp(2g,\BF_p)$ that lie in  $\BL(\BZ)$. 
Also any matrix of the form  $\left[\begin{smallmatrix} V&0 \\ 0 &(V^t)^{-1} \end{smallmatrix} \right]$ has such a lift.
It follows that any element of $U(\BF_p)$ lifts to an element of $\BL(\BZ)$.
\end{proof}

\begin{proof}
[Proof of Lemma \ref{elb2}] Recall that  $\sigma_0=(0,0,\ldots)$ denotes the coloring where all edges are colored zero. 
We are to show that the basis vector $\hat \fb_{\sigma_0}$  in $F_p(g,0,0)$ is fixed up to scalars by  $\BB(\BF_p)$. It will suffice to show that $\hat \fb_{\sigma_0}$  is fixed by $U(\BF_p)$.   
Note that $\hat \fb_{\sigma_0}$ is the reduction modulo $1-\zeta_p$ of the basis vector $\tilde\fb_{\sigma_0}$ of $\BS_p(\Si_g(0))$ which is represented by the  empty skein, and is thus fixed  by any element of $\Gamma_g$ which extends to a diffeomorphism of $\bH_g$, as observed in Remark~\ref{3.3}. 
By Proposition \ref{LZF}, any element of $U(\BF_p)$ lifts to an element of $\BL(\BZ)$, and by Proposition \ref{Lextend}, any element of $\BL(\BZ)$ 
 is induced by an element of $\Gamma_{g}$    which
 extends to a diffeomorphism of $\bH_g$. This implies the result.
\end{proof}

\section{Proof of Theorems~\ref{1.1} and~\ref{1.8}}\lbl{sec4}

Theorem~\ref{1.1} follows easily from  Theorem~\ref{1.8} and the dimension formulae of \cite[p.~229]{GM3}, where we computed the cardinality of the sets $C_p(g,c,\varepsilon)$ in the various cases in terms of the Verlinde formula (\ref{Verl}) and its cousin (\ref{d-Verl}). 

In the proof of Theorem~\ref{1.8}, we shall need one more result from modular representation theory. Recall that the set of simple positive roots for the symplectic Lie algebra
 consists of $\al_1=\varepsilon_1-\varepsilon_2, \ldots, \al_{g-1}=\varepsilon_{g-1}-\varepsilon_g$, and $\al_g=2 \varepsilon_g$ \cite[Planche III]{B}. 

\begin{lem} \lbl{lem-Premet} Let $p>2$ and suppose $\la$ is a $p$-restricted dominant weight for $\Sp(2g,K)$. Let $\Pi(\la)$ be the set (without multiplicities) of weights occuring in the simple module $L_p(\la)$. 
\begin{enumerate}
\item[(i)] If $\mu\in \Pi(\la)$ is such that $\mu + \al_i\not\in \Pi(\la)$ for all $i=1, \ldots, g$, then $\mu = \la$.
\item[(ii)] If $\la= \sum \eta_i \om_i$ and $\eta_i>0$ for some $i=1, \ldots, g$, then $\la -\al_i  \in \Pi(\la)$.
\end{enumerate}
\end{lem}
\begin{proof} This is true for the sets of weights of simple $\Sp(2g,\BC)$-modules. 
By a result of Premet \cite{Pr} (see also the discussion in \cite[\S 3.2]{Hum}), $\Pi(\la)$ is the same when working over $K$ or $\BC$ as long as 
$\la$ is $p$-restricted and 
the characteristic $p>2$. The result follows.
\end{proof}

Let us now prove Theorem~\ref{1.8}.  Recall that we must determine $\widetilde W_p(g,c,\varepsilon)$ (= the multiset of weights occuring in $\widetilde F_p(g,c,\varepsilon)$), and we must determine which of the weights in $\widetilde W_p(g,c,\varepsilon)$ is the highest weight, which we denote by $\la_p(g,c,\varepsilon)$. As the details of this are somewhat involved, let us first outline the strategy of the proof.  The proof proceeds in four steps, as follows.

\medskip\noindent{\bf Step 1.} The first step is to compute $\widetilde W_p(g,c,\varepsilon)$ modulo $p-1$. As observed in Remark~\ref{2.3}, we already know the answer: it is the multiset $W_p(g,c,\varepsilon)$ which was determined in Corollary~\ref{2.2}. Recall that the elements of $W_p(g,c,\varepsilon)$ are reduced weights, and that every coloring in $C_p(g,c,\varepsilon)$ determines a reduced weight in $W_p(g,c,\varepsilon)$.

\medskip\noindent{\bf Step 2.} The second step is to identify 
$\overline\la_p(g,c,\varepsilon)$, that is,   
the reduced weight in $W_p(g,c,\varepsilon)$ which is the reduction modulo $p-1$ of the highest weight. In Case $\mathrm{I}$, this is the reduced weight associated to the zero coloring, as proved in Corollary~\ref{2.4}. In the three other cases, we will show that 
$\overline\la_p(g,c,\varepsilon)$
is the reduced weight associated to
a coloring
illustrated  in Figure \ref{all}.
\begin{figure}[h]
\centerline{\includegraphics[width= .52in]{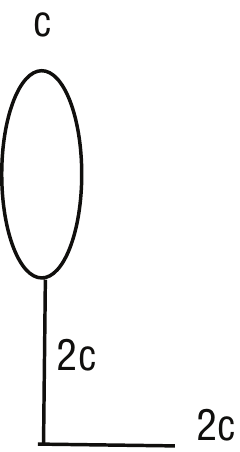}  \quad , \includegraphics[width=1.2in]{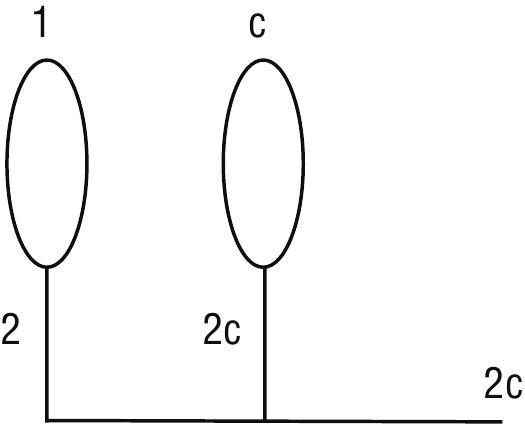}\quad , \includegraphics[width=1.2in]{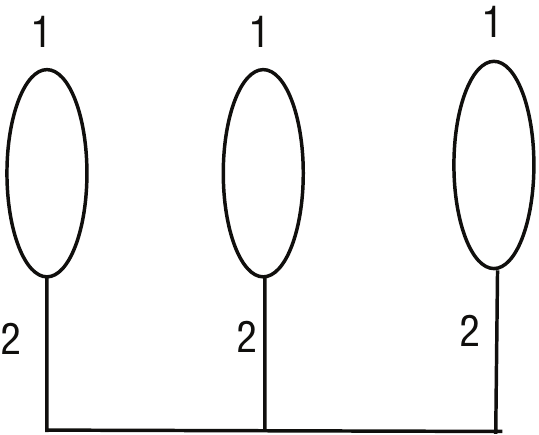}}
\caption{The colorings associated to the highest weights in Case $\mathrm{II}$, Case $\mathrm{III}$, and Case $\mathrm{IV}$. The leftmost part of $G_g$ is not drawn as it is colored zero. For the same reason the 
rightmost
edge is not drawn in Case $\mathrm{IV}$ as $c$ is zero. These graphs could be guessed by taking the smallest coloring which is rightmost on the graph and has the  given type.}\lbl{all}
\end{figure}

This will be proved 
by showing that these colorings satisfy the hypothesis of the 
following Lemma. We let $\overline\al_i$ denote the reduction 
modulo $p-1$ of the root  $\alpha_i$.

\begin{lem}\lbl{4.1} Let $\sigma$ be a coloring in $C_p(g,c,\varepsilon)$ and let $\overline w(\sigma)$ be its associated reduced weight. 
If   $\overline w(\sigma) + \overline\al_i \notin W_p(g,c,\varepsilon)$  for all $i=1, \ldots, g$, then 
$\overline w(\sigma)= \overline \la_p(g,c,\varepsilon)$.
\end{lem}

\begin{proof} Let $\mu$ be any lift of $\overline w(\sigma)$ to $\widetilde W_p(g,c,\varepsilon)$. The hypothesis implies that for all $i=1, \ldots, g$, the weight $\mu+\alpha_i$ does not occur in $\widetilde W_p(g,c,\varepsilon)$. By the construction of the module $\widetilde F_p(g,c,\varepsilon)$, we know that its highest weight is $p$-restricted, so that we can apply Lemma~\ref{lem-Premet}(i). It follows that $\mu$ is the highest weight, and so $\overline w(\sigma)$ is the reduction modulo $p-1$ of the highest weight, as claimed.
\end{proof}

\medskip\noindent{\bf Step 3.} Once 
$\overline \la_p(g,c,\varepsilon)$ is known, the third step will be to determine $ \la_p(g,c,\varepsilon)$.
Write $\la_p(g,c,\varepsilon)=\sum \eta_i \om_i$ with $0\leq \eta_i\leq p-1$ (since $\la_p(g,c,\varepsilon)$ is $p$-restricted). The coefficient $\eta_i$ is determined by its reduction 
 $\overline\eta_i$ modulo $p-1$ provided  $\overline\eta_i\not= 0$. But if $\overline\eta_i=0$, then $\eta_i$ can be either $0$ or $p-1$. We shall show that  $\eta_i=0$ whenever $\overline\eta_i=0$ in all our cases by means of the following Lemma.

 \begin{lem}\lbl{4.4} Let 
$\la_p(g,c,\varepsilon)= \sum \eta_i \om_i$ 
be the highest weight in  $\widetilde W_p(g,c,\varepsilon)$. 
If $\overline\la_p(g,c,\varepsilon) -\overline\al_i \notin W_p(g,c,\varepsilon)$, 
then $\eta_i=0$. 
\end{lem}

\begin{proof} This follows immediately from Lemma~\ref{lem-Premet}(ii).
\end{proof}

\medskip\noindent{\bf Step 4.} Once Step 3 is completed, it only remains to prove that $\widetilde W_p(g,c,\varepsilon)$ is as claimed in the theorem. This is now easy. Recall the notation $d=(p-1)/2$. We simply note that all weights in $\widetilde W_p(g,c,\varepsilon)$ must lie in $[1-d,d-1]^g$ as they must lie in the convex hull of the orbit under the Weyl group of the highest weight, and in each case this highest weight has been shown in Step 3 to lie in $[0,d-1]^g$. But no two distinct integer points in $[1-d,d-1]^g$ agree modulo $p-1=2d$ in each coordinate. Thus $\widetilde W_p(g,c,\varepsilon)$ is determined by its reduction modulo $p-1$, and we are done.

\medskip

In the rest of this section, we shall now  carry out Steps 2 and 3 in the various cases. Having done this, the proof will be complete. To simplify notation, we shall denote the highest weight $\la_p(g,c,\varepsilon)$ simply by $\la$. 
Also, from now on when we say coloring, we mean a small admissible $p$-coloring.
   
Recall that a coloring $\sigma= (a_1, \ldots, a_g, b_1, \ldots, b_g, c_1, \ldots )$ assigns the color $2a_i$ to the $i$th stick edge and the color $a_i+b_i$ to the $i$th loop edge (see Figure~\ref{lol2}).  Recall also that $a_i \ge 0$, $b_i \ge 0$  
satisfy the smallness condition 
$a_i +b_i \le d-1$,
and the coefficient of $\varepsilon_i$ in the weight $w(\si)$ 
is $d-1-a_i-2b_i$ (see (\ref{defw})).

Before we begin with the cases, we state two  lemmas. 
Both are  an easy consequence of Corollary \ref{2.2} and  the smallness condition.
Recall $2d=p-1.$
\begin{lem}\lbl{dcoef} 
       If     $n_i \equiv d \pmod{2d}$
for some  $1 \le  i \le g$, then
$\sum_{i=1}^g n_i \overline \varepsilon_i\notin W_p(g,c,\varepsilon)$,
\end{lem}

\begin{lem}\lbl{d-1} Suppose  
$\sigma= (a_1, \ldots, a_g, b_1, \ldots, b_g, c_1, \ldots )$, 
and 
      $\overline w(\si)=\sum_{i=1}^g n_i \overline \varepsilon_i$.
If 
       $n_i \equiv d-1 \pmod{2d}$
 for some
$1 \le  i \le g$, then $a_i= b_i =0.$ \end{lem}

\medskip
\noindent {\bf  Case $\mathrm{I}$.} Recall $\si_0$ is the coloring which is zero on every edge.   Step 2 was already taken in Corollary  \ref{2.4} and so we 
know that 
$\overline \la = \overline w(\sigma_0)=(d-1) \overline \omega_g$. In Step 3, we must show that $\la =(d-1)\omega_g$. By Lemma \ref{4.4}, it is enough to show that $\overline \la- \overline \al_j \notin W_p(g,0,0)$ for $1 \le j \le g-1$, which follows easily from Lemma \ref{dcoef}.  
This completes the proof in Case  $\mathrm{I}$.

\begin{rem}\lbl{sons}{\normalfont One cannot accomplish Step 2 in Case $\mathrm{I}$ in the same way as we do below in Cases $\mathrm{II}$,  $\mathrm{III}$, and $\mathrm{IV}$, as
 $\sigma_0$ does not satisfy the hypotheses of  Lemma \ref{4.1}.
This is because $\overline w(\sigma_0)+\overline \al_g= (d-1) \overline \om_{g-1}+(d+1)  \overline \varepsilon_g  = \overline w(\sigma') \in  W_p(g,0,0)$, where $\sigma'$  denotes the coloring with all $a_i$'s, $b_i$'s and $c_i$'s  zero except for $b_g= d-1.$  }\end{rem}

\medskip
\noindent {\bf  Case $\mathrm{II}$.}  Let $\sigma$ be the coloring on the left of Figure \ref{all}, then
$$\overline w(\sigma) = (d-1) \sum_{i=1}^{g-1} \overline\varepsilon_i +(d-c-1) \overline \varepsilon_g=
(d-c-1) \overline \om_g+ c \ \overline\om_{g-1}.$$

We  
shall 
show that $\la = w(\sigma)=
(d-c-1)\om_g+ c\,\om_{g-1}.$ 

As explained above, Step 2 in the proof is based on Lemma~\ref{4.1}. We 
must 
show that $\overline w(\sigma) + \overline \al_i \notin W_p(g,c,0)$ for all $1 \le i \le g$. For $i\neq g$ this follows from Lemma \ref{dcoef}. For $i=g$, it is proved by contradiction, as follows. Assume that $\overline w(\sigma) +\overline \al_g
\in  W_p(g,c,0)$. Then there 
 is a  
coloring 
$\si'=(a'_1, \ldots, b'_1, \ldots, c'_1, \ldots)$ of 
type $(c,0)$ with $$\overline w(\si')= 
\overline  w(\sigma) +\overline \al_g =
(d-1) \sum_{i=1}^{g-1} \overline \varepsilon_i 
+(d-c+1) \overline \varepsilon_g~.$$ 
We will see such a coloring is impossible. By  Lemma \ref{d-1}, $a'_i=b'_i=0$ for $i \ne g$. 
In other words, $\si'$ must color all but the rightmost lollipop by zero. By admissibility, it follows that $a'_g=c$.
On the other hand, 
since the coefficient of $\overline\varepsilon_g$ in $\overline w(\si')$ is $ d-1-a'_g-2b'_g\pmod{2d}$, we have 
\begin{equation}\lbl{eqg} 
d-c+1\equiv d-1-a'_g-2b'_g = d-1-c-2b'_g \pmod{ 2d}~,
\end{equation}
so      
 $b'_g \equiv -1 \pmod{ d},$
so $b'_g=d-1$ 
by smallness
and hence $a'_g=0$
again
by smallness (see (\ref{smallness})). 
This contradicts $a'_g=c>0$.  
Thus  $\si'$ does not exist 
and hence 
 $\overline w(\sigma) +\overline \al_g
\notin  W_p(g,c,0)$.

This completes Step 2, and we now know that $\overline \la = \overline w(\si)=
(d-c-1) \overline \om_g+ c \ \overline\om_{g-1}$. 

Next, Step 3 in the proof is based on Lemma~\ref{4.4}. We 
must 
show that $\overline \la - \overline \al_i \notin   W_p(g,c,0)$ whenever the coefficient of $\overline \omega_i$ in $\overline \la$ is zero. For $1 \le j \le g-2$, this follows from Lemma \ref{dcoef}.  
Then if  
$1\leq c\leq d-2$, Step 3 is already complete, as the coefficient of $\overline \omega_g$ and of $\overline \omega_{g-1}$ in $\overline \la$ is non-zero. But if $c=d-1$, the coefficient of $\overline \omega_g$ in $\overline \la$ is zero and we therefore also need to show that $\overline \la - \overline \al_g \notin   W_p(g,d-1,0)$.
This is again proved by contradiction. Assume that $\overline \la -\overline \al_g
\in  W_p(g,d-1,0)$. Then there is a
coloring $\si''$ of 
type $(d-1,0)$ with 
\begin{equation}\lbl{eq''}
\overline w(\si'')= 
\overline \la -\overline \al_g =
(d-1) \sum_{i=1}^{g-1} \overline \varepsilon_i  
      -2 \, \overline \varepsilon_g~.
\end{equation}  
As before when we ruled out the coloring  $\sigma'$ in Step 2, it follows from Lemma \ref{d-1} that $\si''$ must color all but the rightmost lollipop by zero. By admissibility, it follows that $a''_g=c=d-1$
and so $b''_g=0$ by smallness. 
But then the coefficient of $\overline\varepsilon_g$ in $\overline w(\si'')$ must be 
$d-1-a''_g-2b''_g =0\pmod{2d}$, which contradicts (\ref{eq''})
where this coefficient is $-2\pmod{2d}$.
This contradiction shows that $\si''$ does not exist. This completes the proof in Case  $\mathrm{II}$.

\medskip
In Case $\mathrm{III}$ and $\mathrm{IV}$, we will also need the following Lemma which says that odd colorings must assign nonzero colors to at least two 
lollipop sticks, 
and to at least three 
lollipop sticks if $c=0$.

 \begin{lem}\lbl{3-lolli} Suppose  $\sigma= (a_1, \ldots, a_g, b_1, \ldots, b_g, c_1, \ldots )$ is a small admissible $p$-coloring of type 
 $(c,1)$.
\begin{enumerate}\item[(i)] There are at least two distinct $i \in \{1, \ldots, g\}$ with $a_i>0$.
\item[(ii)] If moreover  $c=0$, then there are at least three distinct $i \in \{1, \ldots, g\}$ with $a_i>0$.
\end{enumerate}
\end{lem}
\begin{proof} 
(i) If only one of the $a_i$ is non-zero, say $a_{i_0}$, then admissibility implies that $a_{i_0}=c$ and so condition (\ref{type}) in the definition of colorings of type 
$(c,1)$ is violated.

(ii) If moreover $c=0$, and only two of the $a_i$ are non-zero, say $a_{i_0}$ and $a_{i_1}$, then admissibility implies that $a_{i_0}=a_{i_1}$ and so condition (\ref{type}) is again violated.
\end{proof}

\noindent {\bf  Case $\mathrm{III}$.}  Let $\sigma$ be the coloring in the middle of Figure \ref{all}, then 
$$\overline w(\sigma)= (d-1) \sum_{i=1}^{g-2}\overline  \varepsilon_i +(d-2) \overline \varepsilon_{g-1} +(d-c-1)\overline \varepsilon_{g}= (d-c-1) \overline \om_g+  (c-1) \overline \om_{g-1}+\overline \om_{g-2}.$$
We 
shall  
show that $\la = w(\sigma)=(d-c-1) \om_g+  (c-1) \om_{g-1}+ \om_{g-2}.$

 In Step 2, we 
must  
show that $\overline w(\sigma) + \overline \al_i \notin
       W_p(g,c,1)$                                                      
for all $1 \le i \le g$. For $i\leq g-2$ this follows from Lemma \ref{dcoef}. For $i\in\{g-1,g\}$ it is proved by contradiction, as follows.

\medskip\noindent{(Case $i=g-1$.)} Assume that $\overline w(\sigma) + \overline \al_{g-1} \in W_p(g,c,1)$. Then  
there is a  coloring $\si'$ of 
type $(c,1)$ with 
$$\overline w(\si')= \overline w(\sigma) + \overline \al_{g-1}=(d-1) \sum_{i=1}^{g-1} \overline \varepsilon_i +(d-c-2)\overline \varepsilon_g~.$$ 
Note $\si'$ must color all but 
the rightmost lollipop 
with zeros by Lemma \ref{d-1}, and so $\si'$ 
contradicts  
Lemma~\ref{3-lolli}(i).

\medskip\noindent{(Case $i=g$.)} Assume that $\overline w(\sigma)+ \overline \al_{g} \in W_p(g,c,1)$. Then 
there is a  coloring $\si''$ of 
type $(c,1)$ with 
$$\overline w(\si'')= \overline w(\sigma) + \overline \al_{g}= (d-1) \sum_{i=1}^{g-2} \overline \varepsilon_i +(d-2)  \overline \varepsilon_{g-1}+(d-c+1)  \overline  \varepsilon_{g}~.$$ 
By Lemma \ref{d-1} we have $a''_i =0$ for $1\leq i\leq g-2$. By admissibility, it follows that the three colors meeting at the trivalent vertex at the bottom of the rightmost stick are $2a''_{g-1}, 2 a''_g,$ and $2c$. 
Computing the coefficient of $\overline\varepsilon_{g-1} $,
we have $$d-2\equiv d-1-a''_{g-1}-2 b''_{g-1} \pmod{2d}$$  
or $a''_{g-1}+2  b''_{g-1}\equiv 1\pmod{2d}$. By smallness, it follows that  
$a''_{g-1}=1$. This implies two things. First,  since $\varepsilon=1$ 
({\em i.e.,}  
the coloring is odd), it follows that  $a''_{g}\equiv c \pmod{2}$. Second, by the triangle inequality in the admissibility condition at the trivalent vertex at the bottom of the rightmost stick, it follows that  $c-1 \le a''_g \le c+1$.  One concludes that $a''_g=c$.  
The rest of the proof is now the same as in Case $\mathrm{II}$ Step 2.
Computing 
the coefficient of $\overline\varepsilon_{g} $
exactly as was done there (see (\ref{eqg})), we deduce $b''_g\equiv -1\pmod{d}$, hence $b''_g = d-1$ and so $a''_g=0$ by smallness.
This contradicts $a''_g=c>0$.

Thus Step 2 is complete, and we now know that $\overline \la= \overline w(\si)=(d-c-1) \overline \om_g+  (c-1) \overline \om_{g-1}+\overline \om_{g-2}.$

For Step 3, we must  show that $\overline \la - \overline \al_i \notin   W_p(g,c,0)$ whenever the coefficient of $\overline \omega_i$ in $\overline \la$ is zero. For 
$1 \le i \le g-3$, 
this follows from Lemma \ref{dcoef}. 
Then if 
$2\leq c\leq d-2$, Step 3 is already complete, as the coefficient of $\overline \omega_g$, $\overline \omega_{g-1}$, and $\overline \omega_{g-2}$ in $\overline \la$ is non-zero. But if $c=1$, then we also need to show that $\overline \la - \overline \al_{g-1} \notin   W_p(g,1,1)$, and if $c=d-1$, then we also need to show that $\overline \la - \overline \al_{g} \notin   W_p(g,d-1,1)$. The arguments in these two cases will be given below. Note that for $d=2$ (which corresponds to $p=5$)  
we have both  $c=1$ and $c=d-1$, so that  
we need to use both  arguments.

\medskip\noindent{(Case $c=1$.)} Assume for a contradiction that $\overline \la - \overline \al_{g-1} \in   W_p(g,1,1)$. Then  there is a  
coloring $\si'''$ of 
type $(1,1)$ with $$\overline w(\si''')= 
\overline \la -\overline \al_{g-1} =
(d-1) \sum_{i=1}^{g-2} \overline \varepsilon_i +(d-3)\overline\varepsilon_{g-1}
+(d-1) \overline \varepsilon_g~.$$ 
By Lemma \ref{d-1} $\si'''$ must color all but the $(g-1)$st lollipop by zero, which violates Lemma~\ref{3-lolli}(i). This shows that $\overline \la - \overline \al_{g-1} \notin   W_p(g,1,1)$.

\medskip\noindent{(Case $c=d-1$.)}  Assume for a contradiction that
  $\overline \la - \overline \al_{g} \in   W_p(g,d-1,1)$. Then 
there is a  
coloring $\tilde\si$ of 
type $(d-1,1)$ with 
\begin{equation}\lbl{eqtilde}
\overline w(\tilde\si)= 
\overline \la -\overline \al_{g} =
(d-1) \sum_{i=1}^{g-2} \overline \varepsilon_i +(d-2)\overline\varepsilon_{g-1}
-2 \overline \varepsilon_g~.
\end{equation} 
By the exact same reasoning as when showing that $a''_g=c$ for the coloring $\si''$ in Step 2 (Case $i=g$), we have $\tilde a_g=c$. Hence $\tilde a_g=d-1$ (since we assume $c=d-1$) and so $\tilde b_g=0$ by smallness. 
But then the coefficient of $\overline\varepsilon_{g} $ in $\overline w(\tilde\si)$ must be $ d-1-\tilde a_{g} - 2 \tilde b_{g}=0\pmod{2d}$ which contradicts (\ref{eqtilde}).
This shows that $\overline \la - \overline \al_{g} \notin   W_p(g,d-1,1)$.

\medskip
 The proof in Case  $\mathrm{III}$ is now complete.

\medskip

\noindent {\bf  Case $\mathrm{IV}$.}  Let $\sigma$ be the coloring on the right of Figure \ref{all}, then 
$$\overline w( \si) = (d-1) \sum_{i=1}^{g-3} \overline  \varepsilon_i +(d-2) \sum_{i=g-2}^{g} \overline  \varepsilon_i = (d-2)\overline  \om_g+ \overline  \om_{g-3}.$$

 We 
shall  
show that $\la = w(\sigma)=(d-2)\om_g+ \om_{g-3}.$

 In Step 2, we  
must  
show that $\overline w(\sigma) + \overline \al_i \notin W_p(g,c,0)$ for all $1 \le i \le g$. For $i\leq g-3$ and also for $i=g$, this follows from Lemma \ref{dcoef}. For $i=g-2$, it is proved by contradiction, as follows. Assume that $\overline w( \si)  +\overline  \al_{g-2}\in W_p(g,0,1)$. Then there is a coloring $\si'$ of 
type $(0,1)$ with 
$$\overline w(\si')= \overline w(\sigma) + \overline \al_{g-2}=(d-1) \sum_{i=1}^{g-2} \overline \varepsilon_i +(d-3) \overline \varepsilon_{g-1}+(d-2) \overline  \varepsilon_g~.$$ 
Note that  $\si'$ must color all but two of the lollipops  with zeros by Lemma \ref{d-1} 
and so $\si'$ contradicts Lemma~\ref{3-lolli}(ii). This shows that $\overline w( \si)  +\overline  \al_{g-2}\notin W_p(g,0,1)$.
We refer to this  argument as the ``two lollipop argument''.

For $i=g-1$, the proof is 
similar: 
We have that $$\overline w( \si)  + \overline \al_{g-1}= (d-1) \sum_{i=1}^{g-3} \overline \varepsilon_i +(d-2) \overline \varepsilon_{g-2}+(d-1)\overline  \varepsilon_{g-1}+(d-3) \overline \varepsilon_g $$ and the two lollipop argument shows that $\overline w( \si)  + \overline \al_{g-1} \notin W_p(g,0,1)$.

Thus Step 2 is complete, and we now know that $\overline \la= \overline w(\si)=(d-2)\overline  \om_g+ \overline  \om_{g-3}.$

Concerning Step 3, we have that $\overline \la - \overline \al_i \notin   W_p(g,0,1)$ for $i\leq g-4$ by Lemma \ref{dcoef}, and for $i=g-2$ or $g-1$  by the two lollipop argument. Thus Step 3 is complete, except if $d=2$ 
(hence $p=5$) 
in which case we also need to show that 
$\overline \la - \overline \al_g \notin   
      W_5(g,0,1)$, 
 which is easy and left to the reader.

This completes the proof in Case  $\mathrm{IV}$.

\section{Further Comments}\lbl{sec5}

We elaborate on Remark~\ref{1.6}. Let $\overline C_0$ denote the closure of the fundamental alcove. (See {\em e.g.} \cite[3.5]{Hum}.) A dominant weight $\lambda$ lies in $\overline C_0$ iff $$\langle \la + \rho, \beta^\vee \rangle \leq p~,$$ where $\rho$ is sum of the fundamental weights, and $\beta$ is the highest short root (thus $\beta^\vee$ is the highest root of the dual root system).  By \cite[Planche II]{B}, we have  $$\beta^\vee= \alpha_1^\vee + 2 \alpha_2^\vee + \ldots + 2\alpha_g^\vee~.$$ 
Using $\langle \omega_i, \alpha_j^\vee \rangle = \delta_{ij}$, 
one can check that all the weights $\la$ arising in Theorem~\ref{1.1} lie outside of 
$\overline C_0$ 
({\em i.e.,} one has $\langle \la + \rho, \beta^\vee \rangle >p$) except for the weight 
$\la= (d-2) \om_3$
in rank $g=3$ (for which $\langle \la + \rho, \beta^\vee \rangle =p$).  In fact, as soon as the rank $g\geq 5$,  one uniformly  has 
\begin{equation*}\lbl{alc}\langle \la + \rho, \beta^\vee \rangle = p+ 2g -4
\end{equation*}
for all the weights in our list.   
Thus the distance of our weights  to the fundamental alcove
grows linearly with the rank $g$.

On the other hand, Theorem~\ref{1.1} also holds in rank $g=2$ for those weights where it makes sense ({\em cf.} Remark~\ref{17s}(i)). It turns out that those weights $\la$ in rank $g=2$ all lie in $\overline C_0$. Hence the dimension of $L_p(\la)$ is given by the Weyl character formula, as it is    
a well-known consequence of the linkage principle 
that for dominant weights $\la$ in $\overline C_0$, the simple module  $L_p(\la)$ is isomorphic to the Weyl module 
$\Delta_p(\la)$
(see {\em e.g.} \cite[3.6]{Hum}). We have checked that 
indeed
for $g=2$ our dimension formulae (see Appendix~\ref{App-B}.2) agree with the Weyl character formula. 

A further consistency check is possible in rank $g=3$. In this case, 
although our weights (with one exception) lie outside of $\overline C_0$, the distance to $\overline C_0$ is not too big (one has $\langle \la + \rho, \beta^\vee \rangle \leq p+2$) 
and one can use the Jantzen Sum Formula to compute the formal character of $L_p(\lambda)$. (See \cite[II.8]{J} and references therein. See also the summary in \cite[3.9]{Hum}.) Here is the answer in the case $\varepsilon =0.$ We have $$\la =\la_p(3,c,0)= c\, \om_{2} + (d-1-c) \om_3~.$$ One finds that $L_p(\la)$ is equal to the Weyl module 
      $\Delta_p(\la)$
if $c\in\{0,1\}$, but for $c\geq 2$ there is a short exact sequence $$ 0 \rightarrow \Delta_p(\mu)  \rightarrow \Delta_p(\la)  \rightarrow L_p(\la)\rightarrow 0$$ where 
$$\mu = \la - 2 \om_2 = (c-2) \om_{2} + (d-1-c) \om_3~.$$

In particular 
$\dim L_p(\la)$ can be computed from the Weyl character formula as 

$$\dim L_p(\la) = \dim \Delta_p(\la) - \dim \Delta_p(\mu)$$
and we have checked that our dimension formulae (see Appendix~\ref{App-B}.3) agree with this.

In rank $g\geq 4$, we have not attempted to compute $L_p(\la)$ with the Jantzen Sum Formula. Note that it is easy to see that $L_p(\la)$ can only very rarely be equal to the Weyl module $\Delta_p(\la)$ for a weight $\la$ that arises in Theorem~\ref{1.1}, because of the following observations (the first two of which 
imply the third).

\begin{itemize}

\item By \cite[Corollary 2.10]{GM3},  $\dim L_p(\lambda_p(g,c,\varepsilon))$,
 with $g\ge 3$, $c$, and $\varepsilon$ held fixed and viewed as a function of $p$ is polynomial  of degree $3g-3$.
 
 \item By the Weyl character formula \cite[24.20]{FH}, $\dim \Delta_p(\lambda_p(g,c,\varepsilon))$, with $g \ge 1$,  $c$, and $\varepsilon$ held fixed and viewed as a function of $p$ is polynomial  of degree $g(g+1)/2$. 

 \item For each $g\ge 4$, $c$, and $\varepsilon$, there is a integer $N(g,c,\varepsilon)$ such that for all $p\ge  N(g,c,\varepsilon)$,
$\dim L_p(\lambda_p(g,c,\varepsilon))< \dim \Delta_p(\lambda_p(g,c,\varepsilon))$.\

\end{itemize}

\appendix \section{Proof of Lemma~\ref{lemm-app}} \lbl{App-A}

 In this appendix, we assume some familiarity with Integral TQFT, in particular with the results of \cite[\S 3]{GM2} and \cite{GM3}. See Remark~\ref{3.2new} above for the correspondence between our present notations and those in \cite{GM2} and \cite{GM3}. It is shown in \cite[Cor.~2.5]{GM3} that every mapping class $f\in \Gamma_{g,1}$ is represented on $F_p(\Si_g(2c))$ by a matrix of the form 
$$\left( \begin{array}{cc} \star& 0\\
\star&\star\end{array} \right)$$
with respect to the direct sum decomposition (as vector spaces) 
$$F_p(\Si_g(2c)) =  F_p(g,c,0)\oplus F_p(g,c,1)~.$$ (Here, the top left $\star$ stands for an element of $\End_{\BF_p}(F_p(g,c,0))$, and similarly for the two other $\star$s.) Lemma~\ref{lemm-app} is equivalent to the following 
\begin{lem}\lbl{lemm-app-2} Any $f$ in the Torelli group ${\mathcal I}_{g,1}$ is represented by a matrix of the form $$\left( \begin{array}{cc} 1& 0\\
\star&1\end{array} \right)$$
\end{lem}
To prove this, let us review how the coefficients of this matrix can be 
computed. Throughout this appendix we put $h=1-\zeta_p$. Recall that $\BZ[\zeta_p] / (h) = \BF_p$. We use the basis $\{\hat b_\sigma\}$ of $F_p(\Si_g(2c))$; it
is the reduction modulo $h$ of the orthogonal lollipop basis
$\{\tilde b_\sigma\}$ of $\BS_p(\Si_g(2c))$ constructed in
\cite{GM2}. The basis $\{\tilde b_\sigma\}$ is orthogonal with respect
to the Hopf pairing $((\ , \ ))$ defined in \cite[\S
3]{GM2}, and we can therefore compute matrix coefficients by pairing with the dual basis  $\{\tilde
b_{\sigma}^\star\}$ satisfying  $$((\tilde b_\sigma,\tilde
b_{\sigma'}^\star))= \delta_{\sigma, \sigma'}~.$$
Here $\tilde
b_{\sigma}^\star$ lies in the dual lattice $\BS_p^\sharp(\Si_g(2c)) \subset \BS_p(\Si_g(2c)) \otimes \BQ(\zeta_p)$. 
We will use that 
\begin{equation}\lbl{dl}
\tilde
b_{\sigma}^\star \sim \tilde b_\sigma^\sharp ~,
\end{equation} where $\tilde b_\sigma^\sharp$ is defined in \cite[Cor. 3.4]{GM2} to be a certain power of $h$ times $\tilde b_\sigma$, and $\sim$ means
equality up to multiplication by a unit in $\BZ[\zeta_p]$.  (The power of $h$ depends on $\sigma$.) 

The Hopf pairing is a symmetric $\BZ[\zeta_p]$-valued form on $\BS_p(\Si_g(2c))$ which depends on a choice of Heegaard splitting of $S^3$. Given $x,y \in \BS_p(\Si_g(2c))$, their Hopf pairing is computed skein-theoretically as follows.  Think of $x$ and $y$ as represented by skein elements in the handlebody $\BH_g$. Pick a complementary handlebody $\BH'_g$ such that $$\BH_g\cup_{\Sigma_g}
\BH'_g = S^3~. $$ Then 
\begin{equation}\lbl{hdlb} ((x,y))= \langle x \cup r(y) \rangle~,\end{equation}  where the right hand side of (\ref{hdlb}) is the Kauffman bracket evaluation (at a certain square root of $\zeta_p$) of the skein element in $S^3$ obtained as the union of $x$ in $\BH_g$ and $r(y)$ in $\BH'_g$ where $r$ is a certain identification of $\BH_g$ with $\BH'_g$ (see \cite[\S 3]{GM2} for a more precise definition). 

 Given now a mapping class $f$, the $(\sigma, \sigma')$ matrix coefficient of $f$ acting on $\BS_p(\Si_g(2c))$ is computed skein-theoretically as the evaluation
\begin{equation}\lbl{mco}
\langle \tilde b_\sigma \cup s \cup r( \tilde
b_{\sigma'}^\star)\rangle
\end{equation} in $$S^3=\BH_g\cup_{\Sigma_g} (\Sigma_g \times \I)\cup_{\Sigma_g}
\BH'_g~,$$  where $s$ is a certain skein element in $\Sigma_g(2c) \times \I$ obtained 
from $f$
in the usual way: there is a banded link $L$ in $\Sigma_{g,1} \times \I$ so that surgery on $L$ gives the mapping cylinder 
of $f$;  one then obtains $s$ by replacing each component of $L$ by a certain skein element $\omega_p$ (see \cite[p.~898]{BHMV}) and placing the resulting skein element in $\Sigma_{g,1} \times \I$ into $\Sigma_g(2c) \times \I$ in the standard way. Here $L$ and $s$ are not uniquely determined by $f$, but it is shown in \cite{BHMV} that this procedure is well-defined, and gives the correct answer. (More precisely, it gives the correct answer up to multiplication by a global projective factor which is a power of $\zeta_p$. Here, we can safely ignore this projective ambiguity as we are eventually interested in the matrix of $f$ modulo $h=1-\zeta_p$ only.) 

We are now ready to prove Lemma~\ref{lemm-app-2}. The main idea is that if $f$ lies in the Torelli group ${\mathcal I}_{g,1}$, then $L$ and hence $s$ are 
very 
special, because the mapping cylinder of $f$ can be obtained by $Y_1$-surgery on $\Sigma_{g,1}\times \I$ \cite{Ha,MM,HM}. The notion of $Y_1$-surgery goes back to Matveev \cite{Mat} (who called it Borromean surgery), and then Goussarov \cite{Gou} and Habiro \cite{Ha} (who called it clasper surgery). We refer the reader to \S 5 of the survey \cite{HM} for a good introduction to $Y_1$-surgery and also for more references to the original papers. 

The result we need is stated in \cite[Prop. 5.5]{HM} and can be formulated as follows. There is a certain 6-component banded link $Y$ in a genus 3 handlebody with the following property. For every $f\in {\mathcal I}_{g,1}$, there exists an embedding of a finite disjoint union of, say, $n$ copies of the pair $(\BH_3, Y)$ into $\Sigma_{g,1}\times \I$, giving rise to a $6n$-component banded link $L$ in $\Sigma_{g,1}\times \I$ such that the mapping cylinder of $f$ is obtained by surgery on this banded link $L$.

Cabling each component of $Y$ by $\omega_p$ gives rise to a skein element ${\mathcal Y}_p$ in $\BH_3$, and the skein element $s$ appearing in our computation of matrix coefficients (see (\ref{mco})) is obtained by placing $n$ copies of ${\mathcal Y}_p$ into $\Sigma_g(2c)\times \I$. 

We can view ${\mathcal Y}_p$ as an element of the free $\BZ[\zeta_p]$-module $\BS_p(\Si_3)$. The orthogonal lollipop basis $\{\tilde b_\sigma\}$ of $\BS_p(\Si_3)$ is indexed by colorings of the form $\sigma = (a_1,a_2,a_3, b_1,b_2,b_3)$. 
As before, let $\sigma_0$ denote the zero coloring. The following lemma is the key to proving Lemma~\ref{lemm-app-2}, as it  shows that all but two coefficients of ${\mathcal Y}_p$ in the $\{\tilde b_\sigma\}$ basis are divisible by $h$. 
\begin{lem} \lbl{Yp} ${\mathcal Y}_p - \tilde b_{\sigma_0} = \alpha \, \tilde b_{(1,1,1,0,0,0)} \pmod h$ for some $\alpha\in \BZ[\zeta_p]$.    
\end{lem}

\begin{proof}[Proof of Lemma~\ref{lemm-app-2} from Lemma~\ref{Yp}] We have that $\tilde b_{\sigma_0}$ is represented by the empty link and $\tilde b_{(1,1,1,0,0,0)}$ is $h^{-1}$ times the elementary tripod (see \cite[Fig.~2 on p.~824]{GM1}). Thus $$s=\sum_{k=0}^n h^{-k}s_k~,$$ where $s_k$ is the disjoint union of $k$ elementary tripods embedded in $\Sigma_g(2c)\times\I$. The contribution of $h^{-k} s_k$ to the $(\sigma, \sigma')$ matrix coefficient of $f$ is \begin{equation} \lbl{mco2}
h^{-k} \langle \tilde b_\sigma \cup s_k \cup r( \tilde
b_{\sigma'}^\star)\rangle~.
\end{equation} For $k=0$ this is $\delta_{\sigma, \sigma'}$, as $s_0$ is the empty link. For $k>0$, a straightforward application of the lollipop lemma \cite[Thm.~7.1]{GM1} shows that if $\sigma$ and $\sigma'$ have the same parity, then (\ref{mco2}) is divisible by $h$ (since a tripod has $3$ elementary lollipops). This proves Lemma~\ref{lemm-app-2}.
\end{proof}

\begin{proof}[Proof of Lemma~\ref{Yp}] The 6-component banded link $Y$
  in $\BH_3$ can be described as follows (see \cite[Fig. 5]{HM}). We
  can number the components of $Y$ as $Y_1,Y_2,Y_3,Y_1',Y_2',Y_3'$
  such that the following holds.  The components $(Y_1,Y_2,Y_3)$ lie
  in a ball $B \subset \BH_3$ and form zero-framed Borromean
  rings. For $i=1,2,3$, the component $Y'_i$ is a zero-framed unknot
  going once around the $i$th `hole' of $\BH_3$; moreover $Y'_i$ is
  linked exactly once with $Y_i$, and unlinked with the two other
  $Y_j$s. 

Recall that ${\mathcal Y}_p$ is obtained by cabling each of
  these 6 components by $\omega_p$. For the lemma, we need to compute the coefficients of the basis
  vectors $\tilde b_{(a_1,a_2,a_3, b_1,b_2,b_3)}$ in ${\mathcal Y}_p$. (In fact, except for the zero coloring, we only need these coefficients modulo $h$.)  Such a coefficient is given by the Hopf pairing $$(({\mathcal Y}_p, \tilde b^\star_{(a_1,a_2,a_3,
  b_1,b_2,b_3)}))~.$$ As explained above, this is computed as the Kauffman bracket 
  of a certain skein element in $S^3$. Using now the handle
  slide property of $\omega_p$ (which is at the basis of the
  skein-theoretic construction of TQFT \cite{BHMV}), we can compute this by
  performing surgery on $Y$, which gives back $S^3$, but transforms
  the standard embedding of $\BH'_3$ in $S^3$ into an embedding
  $\varphi$ of $\BH'_3$ in $S^3$ where the three handles of
  $\varphi(\BH'_3)$ are linked as in the Borromean rings. Thus the
  coefficient of $\tilde b_{(a_1,a_2,a_3, b_1,b_2,b_3)}$ in ${\mathcal
    Y}_p$ is the Kauffman bracket evaluation 
\begin{equation}\lbl{cov} \langle \varphi(r(\tilde b^\star_{(a_1,a_2,a_3,
  b_1,b_2,b_3)}))\rangle~.
\end{equation}  For the zero coloring, the basis
  vector $\tilde b^\star_{\sigma_0}$ is represented by the empty
  link, and so (\ref{cov}) evaluates to $1$, as asserted. The following Lemma~\ref{Yp2} shows that (\ref{cov})  
  is divisible by $h$ as soon as $\max(b_i)>0$ or
  $\max(a_i)>1$. Thus it only remains to compute (\ref{cov}) for the colorings with $\max(a_i)=1$ and all $b_j=0$. Assume all $b_j=0$. For $a_1=a_2=a_3=1$, there is nothing to prove, while if one of the $a_i$ is zero, then (\ref{cov}) evaluates to zero (as the two remaining handles are unlinked).  This proves the lemma.
\end{proof} 

For $r\in\BQ$, we define its floor $\lfloor r\rfloor$ to be the largest integer $\leq r$, and its roof $\lceil r\rceil$ to be the smallest integer $\geq r$.  

\begin{lem} \lbl{Yp2} We have that (\ref{cov}) is divisible by $$h^{\lfloor (E_1 + E_2 -E_3)/2\rfloor}$$ where $E_i=a_i+2 b_i$ ($i=1,2,3$), and w. l. o. g. we may assume $E_1 \geq E_2 \geq E_3$. 
\end{lem}

\begin{proof} Put $$B_{(a_1,a_2,a_3,
  b_1,b_2,b_3)} = h^{\lceil(E_1+E_2+E_3)/2\rceil}\,  \tilde b^\star_{(a_1,a_2,a_3,
  b_1,b_2,b_3)}~.$$ Then  $B_{(a_1,a_2,a_3,
  b_1,b_2,b_3)}$ is represented by a linear combination of skein elements in $\BH_3$ with coefficients in $\BZ[\zeta_p]$. (This follows from (\ref{dl}) and the definition of  $\{\tilde b_\sigma^\sharp \}$ in \cite[\S 3]{GM2}.) Below, we say that such a skein element has no denominators. 
Let $\beta$ denote the result of evaluating (\ref{cov}) with $B_{(a_1,a_2,a_3,
  b_1,b_2,b_3)}$ in place of $\tilde b^\star_{(a_1,a_2,a_3,
  b_1,b_2,b_3)}$. Since the three lollipops are linked as in the Borromean rings, we can pull the first two of them apart and consider the third as just some skein element in the complement of the first two. In other words, by cutting $B_{(a_1,a_2,a_3,
  b_1,b_2,b_3)}$ at the midpoint of the edge labelled $2a_3$, we can 
write $\beta$  as the genus two Hopf pairing of two skein elements, say $\frak s$ and $\frak s'$, in a genus $2$ handlebody with one banded point labelled $2a_3$. Moreover, one of these two skein elements, say $\frak s$ (namely the one containing the first two lollipops), is a power of $h$ times the basis element $\tilde b^\star_{(a_1,a_2,
  b_1,b_2,2a_3)}$ of the orthogonal lollipop basis of $\BS_p(\Sigma_2(2a_3))$. Rewriting $\frak s'$ also in this basis and using the formulae in \cite[\S 3]{GM2} for the Hopf pairing in the orthogonal lollipop basis, one finds that $\beta$ is divisible by $h^{E_1 +E_2}$. (Here it is important to observe that both $\frak s$ and $\frak s'$ are skein elements without denominators.) Thus (\ref{cov}) is divisible by $$h^{-\lceil(E_1+E_2+E_3)/2\rceil + E_1 +E_2}= h^{\lfloor (E_1 + E_2 -E_3)/2\rfloor}~.$$ This completes the proof.
\end{proof}

\begin{rem}{\normalfont Other applications of Lemma~\ref{Yp2} include a skein-theoretic construction of Ohtsuki's power series invariant for integral homology 3-spheres, and a Torelli group representation inducing this invariant \cite{M}.
}\end{rem}

     \section{Some polynomial formulae for  dimensions} \lbl{App-B}
 
 In \cite[Prop. 7.7, Prop. 7.8]{GM3} we gave residue formulae for  $D_g^{(2c)}(p)$ and  $ \delta_g^{(2c)}(p)$ valid for $g \ge 1$.  
Using Equation  (\ref{diff}), one can then express  $\dim  L_p(\la)$, for the $\la$ that arise in Theorem \ref{1.1}, 
for specified $g$ as polynomials in $p$ and $c$ using mathematical software. 
Below we write down these polynomials in rank $g=2$, $3$, and $4$.
These formulae hold for 
$p\geq 5$ and
 $1 \le c \le d-1$ where $d=(p-1)/2$.
The first polynomial in each rank is the second polynomial with $c$ set to zero. Similarly the fourth polynomial in each rank (except rank 2) is the third polynomial with $c$ set to zero. As noted in Section~\ref{sec5}, our formulae in rank $g=2$ and also the first one in rank $g=3$ agree with Weyl's character formula.
 The second formula in rank $g=3$ agrees  with Weyl's character formula for $c=1$, but not for $c>1$. \medskip 
  
\addtocounter{subsection}{1}
\subsection{Rank $g=2$}  

    \begin{align*}  
    \dim   L_p\big ( (d-1) \om_2 \big)
    = \frac{1}{24}  (p-1)p(p+1) \hspace{10cm}
      \end{align*}

  \begin{align*} 
    \dim  L_p \big( (d-c-1) \om_2 +c\, \om_1 \big)
    = \frac{1}{24} (c+1) (p+1) (p-2 c-1) (p-c) 
    \hspace{10cm}
      \end{align*}

   \begin{align*} 
  \dim  L_p \big ( (d-c-1) \om_2 +(c-1) \om_1\big ) 
  = \frac{1}{24} c (p-1) (p-2 c-1)
   (p-c-1) 
   \hspace{10cm}
     \end{align*}

 \subsection{Rank $g=3$}

  \begin{align*}  
\dim  L_p \big ( (d-1) \om_3 \big )
  =  \frac{1}{2880}(p-1) p (p+1)^2 (p+2) (p+3) \hspace*{10cm} 
     \end{align*}

\begin{align*}  
\hspace{-.1in} \dim L_p &\big( (d-c-1)  \om_3  +c\ \om_2 \big ) = 
   \frac{1} {2880} (p-2 c-1)
 \Big (p^5 (2 c+1)+p^4 (4 c^2+4 c+7)+ \\& p^3 (-12 c^3-18 c^2+28 c+17)+p^2 (6 c^4+12 c^3-22 c^2-28 c+17)+ \\& 6 p (-2 c^3-3 c^2+c+1)+6 c (c^3+2 c^2-c-2) \Big )
    \end{align*}

     \begin{align*} 
    \dim L_p &\big ( (d-c-1) \om_3 +(c-1) \om_2 +\om_1\big )= 
    \frac{1} {2880} (p-1) (p+1)(p-2 c-1)  \Big  ( p^3 (2 c+1)+ \\& p^2 (4 c^2+4 c-5)+ 6 p (-2 c^3-3 c^2+c+1)+6 c (c^3+2 c^2-c-2) \Big )
    \end{align*}

 \begin{align*} 
    \dim  L_p \big( (d-2) \om_3 \big )=
    \frac{1}{2880}(p-3) (p-2) (p-1)^2 p (p+1)
    \hspace{10cm}
      \end{align*}

 \subsection{Rank $g=4$}

   \begin{align*} 
  \dim  L_p \big ( (d-1) \om_4\big ) =   \frac{1}{120960}(p-1) p (p+1) \left(p^6+37 p^4+142 p^2+36\right)
  \hspace{10cm}
    \end{align*}
  
 \begin{align*} 
 \dim L_p &\big( (d-c-1) \om_4  +c\  \om_3 \big )  =
   \frac 1 {120960} (p+1)(p-2c-1) \Big (p^7 (2 c+1)+2 p^6 c (2 c+1)+ \\& p^5 (-6 c^3-13 c^2+18 c+37)+2 p^4 c (-6 c^3-9 c^2+22 c+38)+\\&p^3 (18 c^5+57 c^4-84 c^3-266 c^2+22 c+142)-6 p^2 c (c^5+6 c^4+c^3-28 c^2-12 c+24)+\\&3 p (2 c^6+12 c^5+5 c^4-50 c^3-37 c^2+32 c+12)-6 c (c^5+3 c^4-5 c^3-15 c^2+4 c+12) \Big )  
   \end{align*}

   \begin{align*}
    \dim L_p &\big ((d-c-1) \om_4 + (c-1)\om_3+ \om_2 \big ) =
  \frac 1 {120960} (p-1)(p-2c-1)   \Big (p^7 (2 c+1)+ \\& p^6 (4 c^2+6 c+2)+p^5 (-6 c^3-5 c^2+26 c-12)-2 p^4 (6 c^4+15 c^3-13 c^2-9 c+13)+ \\& p^3 (18 c^5+33 c^4-132 c^3-148 c^2+164 c+23)-6 p^2 (c^6-14 c^4-8 c^3+33 c^2+16 c-12)+ \\& p (-6 c^6+75 c^4+30 c^3-159 c^2-48 c+36)-6 c (c^5+3 c^4-5 c^3-15 c^2+4 c+12) \Big )    \end{align*}
   
 \begin{align*}  
   \dim  L_p \big( (d-2) \om_4 + \om_1 \big ) = \frac{1}{120960}(p-3) (p-2) (p-1)^2 p (p+1)^2 (p+2) (p+3)
     \hspace{10cm}
     \end{align*}

\end{document}